\documentclass[12pt]{amsart}
\usepackage{amsmath}
\usepackage{amssymb}
\usepackage{amsfonts}
\usepackage{color}

\theoremstyle{plain}

\newtheorem{corollary}{Corollary}

\newtheorem{definition}{Definition}

\newtheorem{lemma}{Lemma}

\newtheorem{proposition}{Proposition}
\newtheorem{remark}{Remark}

\newtheorem{theorem}{Theorem}
\numberwithin{equation}{section}

\begin{document}
\title[Beurling density theorems on a vertical strip]{Beurling density
theorems for sampling and interpolation on the flat cylinder}
\author{Lu\'{\i}s Daniel Abreu}
\address{Faculty of Mathematics \\
University of Vienna \\
Oskar-Morgenstern-Platz 1 \\
1090 Vienna, Austria}
\email{abreuluisdaniel@gmail.com}
\author{Franz Luef }
\address{Department of Mathematics, NTNU Trondheim, 7041 Trondheim, Norway}
\email{franz.luef@ntnu.no}
\author{Mohammed Ziyat}
\address{LAMA, Mohammed V University in Rabat \\
Faculty of Sciences, \\
Rabat, Morocco.}
\email{mohammed.ziyat@fsr.um5.ac.ma}
\thanks{The authors would like to thank Antti Haimi for his input during the
early stages of this work. This research was supported by the Austrian
Science Fund (FWF), P-31225-N32 and 10.55776/PAT8205923.}

\begin{abstract}
We provide a \emph{complete description} of \emph{sampling} and \emph{%
interpolation} sets, in terms of upper and lower Beurling densities, for the
Fock space $\mathcal{F}_{2,\nu }^{\alpha }\left( \mathbb{C}/\mathbb{Z}%
\right) $ of entire functions, quasi-periodic with the Weyl translation of
the horizontal variable. The topological space $\mathbb{C}/\mathbb{Z}$, is
represented by the `flat cylinder', the vertical righ-open strip $%
[0,1)\times \mathbb{R}$. As a by-product, for the space $L_{\nu }^{2}(0,1)$
of $\mathbb{R}$-measurable and $[0,1]$-square-integrable, quasi-periodic
functions with factor $e^{2\pi ik\nu }$, we obtain a full description of
`theta Gabor frames' and `theta Riesz basic sequences' of the form%
\begin{equation*}
\{e^{2i\pi z_{2}t}\theta _{z_{2},z_{1}}(-t,i),z_{1}+iz_{2}\in Z\}\text{,}
\end{equation*}%
\ where $Z$ is a discrete separated set and $\theta _{\alpha ,\beta }$ the
Jacobi theta function.
\end{abstract}

\maketitle

\section{Introduction}

In this paper we consider the problem of sampling and interpolation in a
natural cylindrical setting. A Hilbert Fock space $\mathcal{F}_{2,\nu
}^{\alpha }\left( \mathbb{C}/\mathbb{Z}\right) $ is considered on the
topological space $\mathbb{C}/\Gamma $, geometrically equivalent to the
`flat cylinder' $\Lambda (\mathbb{Z}):\mathbb{=}[0,1)\times \mathbb{R}$,
defined by $e^{-\alpha |z|^{2}}$-weighted entire functions on $\mathbb{C}$,
square integrable on $\Lambda (\mathbb{Z})$, satisfying the Weyl translation
horizontal quasi periodicity condition,%
\begin{equation}
F(z+k)=e^{2\pi ik\nu }e^{\frac{\alpha }{2}k^{2}+\alpha zk}F(z)\text{, \ \ }%
k\in \mathbb{Z}\text{.}  \label{funct-equa}
\end{equation}%
We obtain a \emph{complete description} of \emph{sampling and interpolation}
in terms of a concept of \emph{Beurling density} adapted to the geometry of $%
\Lambda (\mathbb{Z})$. This reveals the real number $\frac{\alpha }{\pi }$
as a sharp critical `Nyquist density': for a separated sequence $Z\subset
\Lambda (\mathbb{Z})$, the condition $D^{-}(Z)>\frac{\alpha }{\pi }$
characterizes \emph{sets of sampling}, while the condition $D^{+}(Z)<\frac{%
\alpha }{\pi }$ characterizes \emph{sets of interpolation}. The results are
equivalent to density conditions for the existence of \emph{Gabor frames}
and \emph{Gabor Riesz basis} of the form%
\begin{equation}
\mathcal{G}_{0}\left( g_{0},Z\right) =\{e^{2i\pi z_{2}t}\theta
_{z_{2},z_{1}}(-t,i)\text{, }z=z_{1}+iz_{2}\in Z\text{, }t\in \mathbb{R}\}%
\text{,}  \label{GaborFrameCylinder}
\end{equation}%
where $g_{0}(t)=2^{\frac{1}{4}}e^{-\pi t^{2}}$ and $\theta _{\alpha ,\beta
}(z,\tau ):=\sum_{k\in \mathbb{Z}}e^{i\pi (k+\alpha )^{2}\tau +2\pi
i(k+\alpha )(z+\beta )},\quad $Im$(\tau )>0$, is the classical Jacobi theta
function, for the space $L_{\nu }^{2}(0,1)$ of $\mathbb{R}$-measurable, $%
[0,1]$-square integrable quasi-periodic functions in the sense, 
\begin{equation}
f(t+k)=e^{2\pi ik\nu }f(t),\quad t\in \mathbb{R}\text{,}\,k\in \mathbb{Z}%
\text{.}  \label{functional}
\end{equation}

\subsection{The Fock space of entire functions}

Let $\mathcal{F}_{2}^{\alpha }(\mathbb{C})$ be the Fock space of entire and
square-integrable functions on $\mathbb{C}$ with respect to the Gaussian
measure. The classical Bargmann transform 
\begin{equation}
\mathcal{B}^{\alpha }f(z)=2^{1/4}\int_{\mathbb{R}}f(t)e^{2\alpha tz-\alpha
t^{2}-\frac{\alpha }{2}z^{2}}dt\text{,}  \label{BarDef0}
\end{equation}%
is well known to be an unitary isomorphism $\mathcal{B}^{\alpha }:{L}^{2}(%
\mathbb{R})\rightarrow \mathcal{F}_{2}^{\alpha }(\mathbb{C})$ \cite[Theorem
3.4.3]{Charly} and to be related to the{\ \emph{short-time Fourier transform
(STFT)}}, defined \cite{Charly} in terms of time-frequency shifts acting on
a window $g\in {L}^{2}(\mathbb{R})$ 
\begin{equation*}
\pi (x,\xi )g(t):=e^{2\pi i\xi t}f(t-x),\qquad (x,\xi )\in {\mathbb{R}^{2}}%
,t\in {\mathbb{R}}\text{,}
\end{equation*}%
as {\ 
\begin{equation}
V_{g}f(x,\xi )=\left\langle f,\pi (x,\xi )g\right\rangle _{L^{2}\left( 
\mathbb{R}\right) }=\int_{%
\mathbb{R}
}f(t)\overline{g(t-x)}e^{-2\pi i\xi t}dt\in L^{2}(\mathbb{R}^{2})\text{.}
\label{Gabor}
\end{equation}%
}When $g$ is specialized to be a Gaussian, $g(t)=g_{0}(t)=2^{\frac{1}{4}%
}e^{-\pi t^{2}}$, (\ref{Gabor}) and (\ref{BarDef0}) are related by the
formula 
\begin{equation}
\mathcal{B}f(z)=\mathcal{B}^{\pi }f(z)=e^{-i\pi x\xi +\frac{\pi }{2}%
\left\vert z\right\vert ^{2}}V_{g_{0}}f(x,-\xi )\text{,}\ z=x+i\xi \text{.}
\label{BargDef}
\end{equation}%
Lyubarskii \cite{Ly} and Seip-Wallst\'{e}n \cite{seip1,seip2}, independently
described the sampling and interpolating sequences in $\mathcal{F}%
_{2}^{\alpha }(\mathbb{C})$. Using the relation (\ref{BargDef}), it is easy
to see that $Z\subset \mathbb{C}$ is a sequence of sampling in $\mathcal{F}%
_{2}^{\alpha }(\mathbb{C})$ exactly when the sequence $Z\subset \mathbb{R}%
^{2}$ generates a Gabor frame $\mathcal{G}\left( g_{0},Z\right) =\{e^{2\pi
iz_{2}t}g_{0}(t-z_{1})$, $z=(x,\xi )\in Z\}$ in ${L}^{2}(\mathbb{R})$ with a
Gaussian window $g_{0}(t)=2^{\frac{1}{4}}e^{-\pi t^{2}}$. This settled a
conjecture of Daubechies and Grossmann \cite{DauGross}, concerning the
density of sets generating Gabor frames for the Gaussian window. See, for
instance \cite{AoP} for more details on this connection in the lattice case.

\subsection{The Fock space of quasi-periodic entire functions}

Let $\nu ,\,\alpha \in \mathbb{R}$ be fixed parameters. The core object of
this paper is $\mathcal{F}_{2,\nu }^{\alpha }\left( \mathbb{C}/\mathbb{Z}%
\right) $, the Fock-type space of entire functions satisfying the
quasi-periodicity condition (\ref{funct-equa}) and equipped with the norm 
\begin{equation*}
\Vert F\Vert _{\mathcal{F}_{2,\nu }^{\alpha }\left( \mathbb{C}/\mathbb{Z}%
\right) }^{2}:=\int_{\mathbb{C}/\mathbb{Z}}|F(z)|^{2}e^{-\alpha |z|^{2}}dA(z)%
\text{.}
\end{equation*}%
This space seems to have been first considered in \cite{GhanmiIntissar2008}.
To integrate over the topological space $\mathbb{C}/\mathbb{Z}$, we consider$%
\ \mathbb{Z}$ as a discrete subgroup of $\mathbb{C}$ and associate to it the
vertical strip 
\begin{equation}
\Lambda (\mathbb{Z}):\mathbb{=}[0,1)\times \mathbb{R}\text{,}  \label{strip}
\end{equation}%
as fundamental domain. In higher dimensions, the analogue object $\mathbb{C}%
^{d}/\Gamma $, where $\Gamma $ is a discrete subgroup of $(\mathbb{C}^{d},+)$
is called a \emph{quasi-torus} \cite{Z1,IZ2,Florentino,Capo}. A sequence $%
Z\subset \lbrack 0,1)\times \mathbb{R}$ is said to be\emph{\ a set of
sampling} for $\mathcal{F}_{2,\nu }^{\alpha }\left( \mathbb{C}/\mathbb{Z}%
\right) $ if there exist constants $A$ and $B$, such that, for every $F\in 
\mathcal{F}_{2,\nu }^{\alpha }\left( \mathbb{C}/\mathbb{Z}\right) $, 
\begin{equation}
A\Vert F\Vert _{\mathcal{F}_{2,\nu }^{\alpha }\left( \mathbb{C}/\mathbb{Z}%
\right) }^{2}\leq \sum_{z\in Z}|F(z)|^{2}e^{-\alpha |z|^{2}}\leq B\Vert
F\Vert _{\mathcal{F}_{2,\nu }^{\alpha }\left( \mathbb{C}/\mathbb{Z}\right)
}^{2}\text{.}  \label{samplingFockdef}
\end{equation}%
In section 2.3, we rephrase definition (\ref{samplingFockdef}) and the
corresponding for $\mathcal{F}_{\infty ,\nu }^{\alpha }\left( \mathbb{C}/%
\mathbb{Z}\right) $ in a form that is more suitable for the proofs, together
with the definition of a sequence of interpolation for $\mathcal{F}_{2,\nu
}^{\alpha }\left( \mathbb{C}/\mathbb{Z}\right) $. In the next section we
will see that (\ref{samplingFockdef}) is equivalent to $Z$\ generating a
frame for the theta-Gabor system (\ref{GaborFrameCylinder}).

\subsection{Quasi-periodic square-integrable functions}

Consider the $L^{2}$-space 
\begin{equation*}
L_{\nu }^{2}(0,1)=\left\{ f:\mathbb{R}\rightarrow \mathbb{C}\ \text{%
measurable with }\left\Vert f\right\Vert _{L^{2}[0,1]}<\infty \text{,
satisfying (\ref{functional})}\right\} \text{.}
\end{equation*}%
The classical Bargmann transform (\ref{BargDef}) is\ an unitary isomorphism $%
\mathcal{B}^{\alpha }:L_{\nu }^{2}(0,1)\rightarrow \mathcal{F}_{2,\nu
}^{\alpha }\left( \mathbb{C}/\mathbb{Z}\right) $ \cite{GhanmiIntissar2013}.
The family of functions $\{e_{k,\nu }(t)=e^{2\pi it(\nu +k)}\}_{k\in \mathbb{%
Z}}$ is an orthonormal basis of $L_{\nu }^{2}(0,1)$. It is possible to
construct functions in $L_{\nu }^{2}(0,1)$ out of functions in Feichtinger's
algebra $\mathcal{S}_{0}(\mathbb{R})$, defined as \cite%
{Fei,FeiZim,FeiModulation}, 
\begin{equation*}
\mathcal{S}_{0}(\mathbb{R}):=\left\{ g\in {L}^{2}(\mathbb{R}):V_{h_{0}}g\in
L^{1}(\mathbb{R}^{2})\right\} \text{,}
\end{equation*}%
by means of the periodization operator 
\begin{equation*}
\Sigma _{\nu }f(t):=\sum_{k\in \mathbb{Z}}e^{2\pi ik\nu }f(t-k)\text{,}
\end{equation*}%
which is well defined as an operator $\Sigma _{\nu }:\mathcal{S}_{0}(\mathbb{%
R})\rightarrow {L}_{\nu }^{2}(0,1)$ with dense range \cite[Proposition 1]%
{ALZ0}. Now, assume $g\in \mathcal{S}_{0}(\mathbb{R})$. Given $f\in {L}_{\nu
}^{2}(0,1)$, a short calculation gives (see \cite[Lemma 2]{ALZ0} for the
details): 
\begin{equation}
\left\langle f,\Sigma _{\nu }\left( \pi (z)g\right) \right\rangle _{{L}_{\nu
}^{2}(0,1)}=\left\langle f,\pi (z)g\right\rangle _{L^{2}\left( \mathbb{R}%
\right) }\text{.}  \label{ident}
\end{equation}%
Since the Bargmann transform\ is an unitary operator $\mathcal{B}:L_{\nu
}^{2}(0,1)\rightarrow \mathcal{F}_{2,\nu }^{\pi }\left( \mathbb{C}/\mathbb{Z}%
\right) $ ,\ then (\ref{ident}) allows to\ rewrite the sampling condition (%
\ref{samplingFockdef}) for the Gaussian window. This is equivalent to saying
that\ there exist constants $A,B>0$ such that, for all $f\in {L}_{\nu
}^{2}(0,1)$,%
\begin{equation}
A\left\Vert f\right\Vert _{{L}_{\nu }^{2}(0,1)}^{2}\leq \sum_{z\in
Z}\left\vert \left\langle f,\Sigma _{\nu }\left( \pi (z)g\right)
\right\rangle _{{L}_{\nu }^{2}(0,1)}\right\vert ^{2}\leq B\left\Vert
f\right\Vert _{{L}_{\nu }^{2}(0,1)}^{2}\text{.}  \label{frame1}
\end{equation}%
Thus, $Z$ being a sampling sequence for $\mathcal{F}_{2,\nu }^{\pi }\left( 
\mathbb{C}/\mathbb{Z}\right) $ is equivalent to the frame property for the
Gabor system in ${L}_{\nu }^{2}(0,1)$, $\mathcal{G}_{\nu }\left(
g_{0},Z\right) $, where 
\begin{equation}
\mathcal{G}_{\nu }\left( g_{0},Z\right) :=\{\Sigma _{\nu }\left( \pi
(z_{1},z_{2})g_{0}\right) ,\;z\in Z\}\text{,}  \label{Frame_v}
\end{equation}%
since $\left\Vert g_{0}\right\Vert _{L^{2}(\mathbb{R})}=1$. Chosing $\nu =0$%
, the Poisson summation formula gives%
\begin{equation*}
\Sigma (\pi (x,\xi )g_{0})(t)=e^{2i\pi \xi t}\sum_{n\in \mathbb{Z}}e^{-\pi
(n+\xi )^{2}+2i\pi (n+\xi -\nu )(-t+x)}=e^{2i\pi \xi t}\theta _{\xi ,x}(-t,i)%
\text{,}
\end{equation*}%
where $\theta _{\alpha ,\beta }$ is the Jacobi theta function (see (\ref%
{classical-theta})). This shows that $\mathcal{G}_{0}\left( g_{0},Z\right) $
is indeed of the form (\ref{GaborFrameCylinder}).

\subsection{Overview, background and outline}

Characterizing frames and sets of sampling is a problem whose classical
origins can be traced back to the work of Duffin and Schaeffer on frames 
\cite{DS} and of Beurling on sampling and interpolation of band-limited
functions \cite{Beurling}. As in the cases of the analytic function spaces
on $\mathbb{C}$ \cite{Ly,seip1,seip2} and on the unit disc \cite{SeipDisc},
we obtain a full description of such sequences. We note that, while the unit
disc is conformally equivalent to the vertical strip $\Lambda (\mathbb{Z})$,
the spaces considered in this paper are not conformally equivalent to the
Bergman spaces on $\mathbb{D}$, to which Beurling scheme has also been
sucessfully adapted by Seip \cite{SeipDisc}. One should also point out that
the results in the disc \cite{SeipDisc} characterize wavelet frames and
Riesz basis with the so-called Cauchy wavelet (the Fourier transform of the
first Laguerre function).

When $Z$ is the regular lattice $Z_{\beta }:=(0,\beta n)\equiv i\beta 
\mathbb{Z\subset }\Lambda (\mathbb{Z}):=[0,1)\times \mathbb{R}$, Gabor
systems in ${L}_{\nu }^{2}(0,1)$ enjoy properties reminiscent of classical
Gabor systems in $L^{2}(\mathbb{R})$. This has been explored in the
companion paper \cite[Section 4]{ALZ0}. The methods of \cite[Section 4]{ALZ0}%
, while providing as a sufficient condition $\beta <1$ for general windows
in $\mathcal{S}_{0}(\mathbb{R})$, are constrained to this regular case. The
choice $g_{0}(t)=2^{\frac{1}{4}}e^{-\pi t^{2}}$ allows to go well beyond the
regular case, since it allows to rephrase the problem in $\mathcal{F}_{2,\nu
}^{\alpha }\left( \mathbb{C}/\mathbb{Z}\right) $, where one can make use of
the powerfull machinery of entire functions. Our \textquotedblleft Nyquist
density" reduces to $\beta <1$ in the regular case with $\alpha =\pi $.\ 

Flat cylindric STFT phase spaces have been recently used in an approach to
the problem of uniqueness of STFT phase retrieval \cite{LukPhi} for signals
defined in intervals of the form $[-\alpha /2,\alpha /2]$, leading to the
flat cylinder phase space $[-\alpha /2,\alpha /2)\times \mathbb{R}$. The
same problem was considered, for functions bandlimited to $\left[ -2B,2B%
\right] $, in \cite{AlaiWell}, leading to the horizontal flat cylinder
domain $\mathbb{R}\times \lbrack -2B,2B)$ as phase space. More comments
about this can be found in \cite[Section 4.4]{ALZ0}

Our description is complete and reveals as a critical `Nyquist density' the
real number $\frac{\alpha }{\pi }$, meaning that \emph{the condition }$%
D^{-}(Z)>\frac{\alpha }{\pi }$\emph{\ characterizes sets of sampling}, while 
\emph{the condition }$D^{+}(Z)<\frac{\alpha }{\pi }$\emph{\ characterizes
sets of interpolation}. According to the correspondences outlined in the
introduction paragraph above, this corresponds to saying that $\mathcal{G}%
_{\nu }\left( h_{0},Z\right) $, the Gabor system (\ref{Gabor}) is a \emph{%
frame} for ${L}_{\nu }^{2}(0,1)$ if and only if $D^{-}(Z)>1$ and it is a 
\emph{Riesz basic sequence} for ${L}_{\nu }^{2}(0,1)$ if and only if $%
D^{+}(Z)>1$\ (see Section 2.3 for precise definitions of the concepts
invoked in this paragraph). In contrast with more elaborated cylindrical
settings \cite{Varolin,Varolin1}, where the description of interpolation is
not complete (due to the punctured Riemann surface structure, the
description lacks the case of equality in the necessary condition for
interpolation), our density inequalities for interpolation are strict.

In the companion paper \cite{ALZ0}, the basic theory of the Short Time
Fourier Transform (STFT) acting on functions satisfying (\ref{funct-equa})
has been developed. Resembling the case of general real-valued functions 
\cite{Abr2010}, Gaussian windows lead to Fock spaces of analytic functions
on $\mathbb{C}/\mathbb{Z}$, Hermite functions give eigenspaces of the Landau
operator on $\mathbb{C}/\mathbb{Z}$, and a vectorial STFT with vectorial
window constituted by the first $n$ Hermite functions, is linked to Fock
spaces of polyanalytic functions on $\mathbb{C}/\mathbb{Z}$.

The proofs follow Beurling's approach \cite{Beurling}, which has been
previously adapted to the Fock space \cite{seip1,seip2} and to the Bergman
space \cite{SeipDisc}. A key step is the construction of functions with
proper growth vanishing exactly on a prescribed set of points. In the Fock
and Bergman spaces, constructions involving Weierstrass or Blaschke products
have been used, but here we introduce new functions vanishing exactly on
prescribed sets of points in the region (\ref{strip}). Some of the proofs in 
\cite{Beurling} and \cite{seip1,seip2} carry over verbatim to our case and
we will not repeat them here. Other steps require some adaptations and some
issues arise due to the geometry of the vertical strip, which is only
invariant by translations by $z$ if $\mathrm{Im}(z)\in \frac{\pi }{\alpha }%
\mathbb{Z}$.

At some points we diverge from the scheme of \cite{Beurling,seip1} and
slightly simplify the proof of the necessary part of the sampling theorem in 
\cite{seip1}. Since every sequence of sampling in the Hilbert space $%
\mathcal{F}_{2,\nu }^{\alpha }\left( \mathbb{C}/\mathbb{Z}\right) $ is a
sampling sequence of its $p=\infty $ version $\mathcal{F}_{\infty ,\nu
}^{\alpha }\left( \mathbb{C}/\mathbb{Z}\right) $, it is only required to
prove that a sequence of sampling in $\mathcal{F}_{\infty ,\nu }^{\alpha
}\left( \mathbb{C}/\mathbb{Z}\right) $ is separated to conclude the same for 
$\mathcal{F}_{2,\nu }^{\alpha }\left( \mathbb{C}/\mathbb{Z}\right) $. It
turns out that this is easier to prove for $\mathcal{F}_{\infty ,\nu
}^{\alpha }\left( \mathbb{C}/\mathbb{Z}\right) $ than for $\mathcal{F}%
_{2,\nu }^{\alpha }\left( \mathbb{C}/\mathbb{Z}\right) $. Also, our proof
that sampling sets for $\mathcal{F}_{2,\nu }^{\alpha }\left( \mathbb{C}/%
\mathbb{Z}\right) $ are sets of uniqueness for $\mathcal{F}_{\infty ,\nu
}^{\alpha }\left( \mathbb{C}/\mathbb{Z}\right) $ (Proposition \ref%
{sam-to-uni}) is a bit more direct than the one in \cite{seip1}. Moreover,
Seip's arguments in \cite{seip1,SeipDisc} rely on the orthogonality on
concentric discs of the eigenfunctions of localization operators \cite{Seip0}%
. We could not find corresponding objects in our setting and had to
circumvent the argument.

The paper can be simply outlined. In the next section we collect the
essential concepts about Bargmann-Fock spaces in $\mathbb{C}/\mathbb{Z}$ and
introduce the notions of lower and upper Beurling densities of $Z\subset
\lbrack 0,1)\times \mathbb{R}$, required to state the main results of the
paper, Theorem \ref{MainSampling} on sampling and Theorem \ref%
{MainInterpolation} on interpolation. Then, in section 3 and section 4, we
give the proofs of the sampling and interpolation results, respectively.

\section{Concepts}

\subsection{The Bargmann-Fock space on a vertical strip}

Consider a lattice $\Gamma =\mathbb{Z}\omega $, $\omega \in \mathbb{C}%
\setminus \{0\}$, and the Fock space of $\Gamma $-quasi-periodic analytic
functions (analytic functions belonging to the quotient group $\mathbb{C}%
/\Gamma $), normed by integration with the Fock weight over a fundamental
domain $\Lambda (\Gamma )$ (a set\ representing $\mathbb{C}/\Gamma $, whose $%
\Gamma $-translations are disjoint and cover $\mathbb{C}$). Without loss of
generality, we may set $\omega =1$. With this simplification, $\Gamma =%
\mathbb{Z}$, $\Lambda (\mathbb{Z})=[0,1)\times \mathbb{R}$, and $\mathcal{F}%
_{2,\nu }^{\alpha }(\mathbb{C}/\mathbb{Z})$ is the space defined in the
Introduction. One can check directly that $\mathcal{B}\left( e_{k,\nu
}\right) =e^{\frac{\alpha }{2}z^{2}+2i\pi (k+\nu )z}$, and this provides the
space $\mathcal{F}_{2,\nu }^{\alpha }\left( \mathbb{C}/\mathbb{Z}\right) $
with an orthogonal basis (see also \cite{GhanmiIntissar2013} or \cite[%
Section 3]{ALZ0}), for which we will use the notation 
\begin{equation}
\varphi _{k}(z):=e^{\frac{\alpha }{2}z^{2}+2i\pi (k+\nu )z}\text{,}\quad
k\in \mathbb{Z}\text{.}  \label{basisF}
\end{equation}%
Moreover, the space $\mathcal{F}_{2,\nu }^{\alpha }\left( \mathbb{C}/\mathbb{%
Z}\right) $ has a reproducing kernel given by{%
\begin{equation*}
K_{\mathcal{F}_{2,\nu }^{\alpha }\left( \mathbb{C}/\mathbb{Z}\right)
}(z,w)=\sum_{k\in \mathbb{Z}}\frac{\varphi _{k}(z)\overline{\varphi _{k}(w)}%
}{\Vert \varphi _{k}\Vert _{\mathcal{F}_{2,\nu }^{\alpha }\left( \mathbb{C}/%
\mathbb{Z}\right) }^{2}}=\sqrt{\frac{2\alpha }{\pi }}e^{\frac{\alpha }{2}%
(z^{2}+\overline{w}^{2})}\theta _{\nu ,0}\left( z-\overline{w},\frac{2i\pi }{%
\alpha }\right) \text{.}
\end{equation*}%
} Here $\theta _{\alpha ,\beta }$ is the Jacobi theta function. 
\begin{equation}
\theta _{\alpha ,\beta }(z,\tau ):=\sum_{k\in \mathbb{Z}}e^{i\pi (k+\alpha
)^{2}\tau +2\pi i(k+\alpha )(z+\beta )},\quad \text{Im}(\tau )>0\text{.}
\label{classical-theta}
\end{equation}

\begin{remark}
The reproducing kernel $K_{\mathcal{F}_{2,\nu }^{\alpha }\left( \mathbb{C}/%
\mathbb{Z}\right) }$ can also be obtained by the periodization principle.
Namely, let $K(z,w)=\frac{\alpha }{\pi }e^{z\overline{w}}$ be the
reproducing kernel of the classical Fock space, then 
\begin{equation*}
K_{\mathcal{F}_{2,\nu }^{\alpha }\left( \mathbb{C}/\mathbb{Z}\right) }(z,w)=%
\frac{\alpha }{\pi }e^{\alpha z\overline{w}}\sum_{k\in \mathbb{Z}}e^{-2\pi
ik\nu }e^{-\alpha zk+\alpha k\overline{w}-\frac{\alpha }{2}k^{2}}\text{.}
\end{equation*}
\end{remark}

We consider also the Banach space $\mathcal{F}_{\infty ,\nu }^{\alpha
}\left( \mathbb{C}/\mathbb{Z}\right) $, the space of holomorphic functions
on $\mathbb{C}$ satisfying the functional equation (\ref{funct-equa}),
equiped with the norm%
\begin{equation*}
\left\Vert F\right\Vert _{\mathcal{F}_{\infty ,\nu }^{\alpha }\left( \mathbb{%
C}/\mathbb{Z}\right) }=\sup_{z\in \mathbb{C}/\Gamma }|F(z)|e^{-\frac{\alpha 
}{2}|z|^{2}}=\sup_{z\in \lbrack 0,1)\times \mathbb{R}}|F(z)|e^{-\frac{\alpha 
}{2}|z|^{2}}\text{.}
\end{equation*}%
At some point in the argument we will also use the (classical) Banach space $%
\mathcal{F}_{\infty ,\nu }^{\alpha }\left( \mathbb{C}\right) $, of
holomorphic functions on $\mathbb{C}$ equipped with the norm%
\begin{equation*}
\left\Vert F\right\Vert _{\mathcal{F}_{\infty }^{\alpha }\left( \mathbb{C}%
\right) }=\sup_{z\in \mathbb{C}}|F(z)|e^{-\frac{\alpha }{2}|z|^{2}}\text{.}
\end{equation*}%
If $F\in \mathcal{F}_{\infty ,\nu }^{\alpha }\left( \mathbb{C}/\mathbb{Z}%
\right) $, then 
we have the following identity, which is going to be useful in the proof of
interpolation:%
\begin{equation}
\left\Vert F\right\Vert _{\mathcal{F}_{\infty ,\nu }^{\alpha }\left( \mathbb{%
C}/\mathbb{Z}\right) }=\left\Vert F\right\Vert _{\mathcal{F}_{\infty
}^{\alpha }\left( \mathbb{C}\right) }\text{.}  \label{normequality}
\end{equation}

\subsection{Beurling density and sets of sampling and interpolation}

A sequence $Z=\{z_{k}\}\subset \Lambda (\mathbb{Z})$ is said to be separated
if 
\begin{equation*}
\delta :=\inf_{j\neq k}\left\vert z_{j}-z_{k}\right\vert >0\text{.}
\end{equation*}%
The constant $\delta $ is called the separation constant of $Z$. Given a
sequence $Z$, we write 
\begin{equation*}
\Vert F|Z\Vert _{2,\alpha }=\left( \sum_{z\in Z}|F(z)|^{2}e^{-\alpha
|z|^{2}}\right) ^{1/2}
\end{equation*}%
and%
\begin{equation*}
\Vert F|Z\Vert _{\infty ,\alpha }=\sup_{z\in Z}|F(z)|e^{-\frac{\alpha }{2}%
|z|^{2}}\text{.}
\end{equation*}%
The definition of a sampling sequence in $\mathcal{F}_{2,\nu }^{\alpha
}\left( \mathbb{C}/\mathbb{Z}\right) $ contains, as first requirement, the
existence of a constant $M>0$ such that the upper inequality holds 
\begin{equation}
M\Vert F|Z\Vert _{2,\alpha }\leq \Vert F\Vert _{\mathcal{F}_{2,\nu }^{\alpha
}\left( \mathbb{C}/\mathbb{Z}\right) }\text{.}  \label{upper}
\end{equation}%
In our proofs, this inequality will hold by the assumptions made in the
results, because a standard argument can be used to show that this is the
case if and only if $Z$ can be written as a finite union of separated sets
(see, e. g. Lemma 7.1 in \cite{seip1}).

\begin{definition}
We will use $M_{2}^{-}(Z,\alpha )$ to denote the smallest constant $M$ such
that%
\begin{equation*}
\Vert F\Vert _{\mathcal{F}_{2,\nu }^{\alpha }\left( \mathbb{C}/\mathbb{Z}%
\right) }\leq M\Vert F|Z\Vert _{2,\alpha }\text{,}
\end{equation*}%
while $M_{\infty }(Z,\alpha )$ is the smallest constant such that 
\begin{equation*}
\Vert F\Vert _{\mathcal{F}_{\infty ,\nu }^{\alpha }\left( \mathbb{C}/\mathbb{%
Z}\right) }\leq M\Vert F|Z\Vert _{\infty ,\alpha }\text{.}
\end{equation*}%
A sequence $Z$ is sampling for $\mathcal{F}_{\infty ,\nu }^{\alpha }\left( 
\mathbb{C}/\mathbb{Z}\right) $ if $M_{\infty }(Z,\alpha )<\infty $. It is
sampling for $\mathcal{F}_{2,\nu }^{\alpha }\left( \mathbb{C}/\mathbb{Z}%
\right) $ if $M_{2}^{-}(Z,\alpha )<\infty $ and the upper bound $%
M_{2}^{+}(Z,\alpha )\Vert F|Z\Vert _{2,\alpha }\leq \Vert F\Vert _{\mathcal{F%
}_{2,\nu }^{\alpha }\left( \mathbb{C}/\mathbb{Z}\right) }$ also holds. $%
M_{2}^{-}(Z,\alpha )$ is called the lower sampling constant and $%
M_{2}^{+}(Z,\alpha )$ is called the upper sampling constant.
\end{definition}

\begin{remark}
One can in many cases find some constant satisfying the inequality in the
above definition, but the determination of the optimal constant $%
M_{2}^{-}(Z,\alpha )$ and $M_{2}^{+}(Z,\alpha )$ is a very challenging
problem. See \cite{FauShaf}, for an overview and a thorough discussion of
the case of sampling sequences in $\mathcal{F}_{2}^{\alpha }\left( \mathbb{C}%
\right) $ with rectangular lattices.
\end{remark}

\begin{definition}
A sequence $Z$ is said to be interpolating for $\mathcal{F}_{2,\nu }^{\alpha
}\left( \mathbb{C}/\mathbb{Z}\right) $, if there exists a constant $C$ such,
for every sequence $(a_{k})_{k\in \mathbb{Z}}$ such that%
\begin{equation*}
\left( a_{k}e^{-\frac{\alpha }{2}|z_{k}|^{2}}\right) _{k\in \mathbb{Z}}\in
\ell ^{2}(\mathbb{Z})\text{,}
\end{equation*}%
one can find an interpolating function $F\in \mathcal{F}_{2,\nu }^{\alpha
}\left( \mathbb{C}/\mathbb{Z}\right) $, such that$\ F(z_{m})=a_{m}$ for all $%
z_{m}\in Z$ and%
\begin{equation}
\Vert F\Vert _{\mathcal{F}_{2,\nu }^{\alpha }\left( \mathbb{C}/\mathbb{Z}%
\right) }\leq C\Vert F|Z\Vert _{2,\alpha }\text{.}  \label{stab}
\end{equation}%
Then $N_{2}(Z,\alpha )$ is defined to be the smallest constant $C$ in (\ref%
{stab}).
\end{definition}

The set $I_{w,r}$, defined below, will play in the geometry of the vertical
strip, the same as the disc $D_{r}(w)$ does in the Euclidean geometry of $%
\mathbb{C}$: 
\begin{equation}
I_{w,r}=[0,1)\times \lbrack \mathrm{Im}(w)-\frac{r}{2},\mathrm{Im}(w)+\frac{r%
}{2}]\text{.}  \label{Idisc}
\end{equation}

\begin{definition}
For a given sequence $Z$ of distinct numbers in $\Lambda (\Gamma )$, let $%
n(Z,I_{w,r})$ denote the number of points in $Z\cap I_{w,r}$. Then the lower
and upper Beurling densities of $Z$ are given, respectively, by 
\begin{equation*}
D^{-}(Z)=\liminf_{r\rightarrow +\infty }\inf_{w\in \Lambda(\Gamma) }\frac{%
n(Z,I_{w,r})}{r}
\end{equation*}%
and 
\begin{equation*}
D^{+}(Z)=\limsup_{r\rightarrow +\infty }\sup_{w\in \Lambda(\Gamma) }\frac{%
n(Z,I_{w,r})}{r}\text{.}
\end{equation*}
\end{definition}

We point out that this\ density is different from the one defined for
cylindrical measures in the punctured plane in \cite{Varolin,Varolin1},
which is obtained from the covering space, while this one is adapted to the
fundamental domain of the cylindrical surface.\ We are now in the position
to state the main results of this paper.

\begin{theorem}
\label{MainSampling}A sequence $Z$ is sampling for $\mathcal{F}_{2,\nu
}^{\alpha }\left( \mathbb{C}/\mathbb{Z}\right) $ if and only if it can be
expressed as a finite union of separated sets and contains a separated
sequence $Z^{\prime }$ such that $D^{-}(Z^{\prime })>\frac{\alpha }{\pi }$.
\end{theorem}

\begin{theorem}
\label{MainInterpolation}A sequence $Z$ is interpolating for $\mathcal{F}%
_{2,\nu }^{\alpha }\left( \mathbb{C}/\mathbb{Z}\right) $ if and only if it
is separated and $D^{+}(Z)<\frac{\alpha }{\pi }$.
\end{theorem}

Observe that, according to the discussion in the Introduction (see \cite%
{ALZ0} for a detailed treatment), this is equivalent to the following
statement.

\begin{corollary}
Let $Z$ be a sequence in $[0,1)\times \mathbb{R}$. The Gabor system with $%
g_{0}(t)=2^{\frac{1}{4}}e^{-\pi t^{2}}$ in $L_{\nu }^{2}(0,1)$, 
\begin{equation*}
\mathcal{G}_{\nu }\left( g_{0},Z\right) =\{\Sigma _{\nu }\left( \pi
(z)g_{0}\right) \}=\{e^{2i\pi z_{2}t}\theta _{\nu -z_{2},z_{1}}(-t,i)\},%
\text{ \ }z=(z_{1},z_{2})\in Z\text{,}
\end{equation*}%
is a Gabor frame for $L_{\nu }^{2}(0,1)$ if and only if $Z$ can be expressed
as a finite union of separated sets and contains a separated sequence $%
Z^{\prime }$ such that $D^{-}(Z^\prime)>1$.
\end{corollary}

\begin{corollary}
Let $Z$ be a sequence in $[0,1)\times \mathbb{R}$. The Gabor system $%
\mathcal{G}_{\nu }\left( g_{0},Z\right) $ is a Gabor Riesz sequence if and
only $Z$ is separated and $D^{+}(Z)<1$.
\end{corollary}

In the course of proving the sampling results for $\mathcal{F}_{2,\nu
}^{\alpha }\left( \mathbb{C}/\mathbb{Z}\right) $, we will also show that a
sequence $Z$ is sampling for $\mathcal{F}_{2,\nu }^{\alpha }\left( \mathbb{C}%
/\mathbb{Z}\right) $ if and only if it contains a separated sequence $%
Z^{\prime }$ such that $D^{-}(Z^{\prime })>\frac{\alpha }{\pi }$.

\begin{remark}
One can rephrase the problem considered in this paper in terms of several
variables, leading to the quasi-tori setting \cite{Ziyat,IZ2}. This seems to
be an extremely difficult problem, even in the Fock space of entire
functions. Partial results are known for sampling on lattices \cite{Groe2011}
and progress has been achieved on sampling and interpolation on lattices, by
replacing the notion of Beurling density (the inverse of the area of the
lattice domain) in terms of the Seshadri constant \cite{LueWang}. Some of
the results on the flat tori have been sucessfully extended to several
dimensions in \cite{JohannesTorus}.
\end{remark}

\begin{remark}
Other variants of the problem can be considered. Sufficient conditions for
Fock spaces of polyanalytic functions on $\mathbb{C}/\mathbb{Z}$ are
equivalent to Gabor frames with Hermite-theta functions \cite[Section 4]%
{ALZ0}, following a similar correspondence as in the $L^{2}\left( \mathbb{R}%
\right) $ case and a sufficient conditions has been obtained in \cite[%
Theorem 1]{ALZ0}, as an analogue of a result of Gr\"{o}chenig and Lyubarskii 
\cite{CharlyYura,CharlyYurasuper}, after uncovering structural properties of
Gabor frames in $L_{\nu }^{2}(0,1)$ reminiscent of the structure of Gabor
frames in $L^{2}\left( \mathbb{R}\right) $ \cite[Section 7]{Charly}. We
expect that an adaptation of the methods developed in this paper allows for
a treatment of sampling with derivatives \cite{BS93}, but we have so far not
pursued this direction.
\end{remark}

\section{Sampling}

\subsection{Necessary conditions for sampling}

Our main results in this subsection are the following necessary density
conditions for sets of sampling.

\begin{theorem}
\label{cct} If a sequence $Z$ is sampling for $\mathcal{F}_{2,\nu }^{\alpha
}\left( \mathbb{C}/\mathbb{Z}\right) $, then it can be expressed as a finite
union of separated sets and contains a separated sequence $Z^{\prime }$ such
that $D^{-}(Z^{\prime })>\frac{\alpha }{\pi }$.
\end{theorem}

The proof of Theorem \ref{cct} requires the following theorem:

\begin{theorem}
\label{SampInf} If a sequence $Z$ is sampling for $\mathcal{F}_{\infty ,\nu
}^{\alpha }\left( \mathbb{C}/\mathbb{Z}\right) $, then it contains a
separated sequence $Z^{\prime }$ such that $D^{-}(Z^{\prime })>\frac{\alpha 
}{\pi }$.
\end{theorem}

For a given closed set $Q$ and $t>0$, let $Q(t)=\{z:\left\vert
z-Q\right\vert <t\}$. The Fr\'{e}chet distance between two closed sets $Q$
and $R$, denoted by $[Q,R]$, is the smallest number $t$ such that $Q\subset
R(t)$ and $R\subset Q(t)$.

\begin{definition}
A sequence of closed sets $Q_{j}$ is said to converge strongly to $Q$,
denoted by $Q_{j}\rightarrow Q$, if $[Q,Q_{j}]\rightarrow 0$. Such a
sequence converges weakly to $Q$ if for every compact set $K$, $\left(
Q_{j}\cap K\right) \rightarrow \left( Q\cap K\right)$.
\end{definition}

We have the following lemma:

\begin{lemma}
\label{limits}Let $Z_{n}$ be a sequence in $\mathbb{C}/\Gamma $ such that $%
\inf \delta (Z_{n})>0$. Then $Z_{n}\rightarrow Z$ implies 
\begin{equation*}
M_{2}(Z,\alpha )\leq \liminf M_{2}(Z_{n},\alpha )\text{.}
\end{equation*}
\end{lemma}

\begin{proof}
The proof is similar to the proof of Lemma 7.3 in \cite{seip1}. In fact, let 
$\varepsilon > 0$. There exists a unit vector $f$ in $\mathcal{F}_{2,\nu
}^{\alpha} \left( \mathbb{C}/\mathbb{Z}\right)$ such that 
\begin{equation*}
\|f|Z\|_{2,\alpha} \leq M^{-1} + \varepsilon,
\end{equation*}
where $M = M_2^-(Z, \alpha)$ (which may be infinite).

For $R>0$ sufficiently large, there exists $N\in\mathbb{N}$ such that for
any $n > N$ we have 
\begin{equation*}
\|f|Z_n \cap [0,\frac1R]\times[-\frac R2,\frac R2]\|_{2,\alpha} \leq M^{-1}
+ C\varepsilon.
\end{equation*}
Where $C$ is independent of $R$. Letting $R \to \infty$, we obtain 
\begin{equation*}
\|f|Z_n\|_{2,\alpha} \leq M^{-1} + C\varepsilon.
\end{equation*}
The result follows since $\varepsilon$ is arbitrary.
\end{proof}

Let $W(Z)$ be the set of all collections of weak limits of all translates $%
Z+z$ of $Z$ such that $\mathrm{Im}(z)\in \frac{\pi }{\alpha }\mathbb{Z}$.
Since the Weyl operators defined by 
\begin{equation}
T_{w}F(z):=e^{\alpha z\overline{w}-\frac{\alpha }{2}|w|^{2}}F(z-w),\quad 
\mathrm{Im}(w)\in \frac{\pi }{\alpha }\mathbb{Z},  \label{weyl}
\end{equation}%
are isometric on $\mathcal{F}_{2,\nu }^{\alpha }\left( \mathbb{C}/\mathbb{Z}%
\right) $, and act as translations on $\mathcal{F}_{2,\nu }^{\alpha }\left( 
\mathbb{C}/\mathbb{Z}\right) $, we conclude that 
\begin{equation*}
M_{p}(Z+z)=M_{p}(Z)\text{,}
\end{equation*}%
and thus%
\begin{equation*}
N_{p}(Z+z)=N_{p}(Z)\text{,}
\end{equation*}%
for all $z\in \mathbb{C}$ such that $\mathrm{Im}(z)\in \frac{\pi }{\alpha }%
\mathbb{Z}$, $p=2$ or $\infty$.

\begin{remark}
Note that a Weyl operator $T_{a}$, defined in \eqref{weyl}, preserves the
functional equation \eqref{funct-equa} only when $\mathrm{Im}(a)\in \frac{%
\pi }{\alpha }\mathbb{Z}.$
\end{remark}

If every $Z_{n}\in W(Z)$ is sampling for $\mathcal{F}_{2,\nu }^{\alpha
}\left( \mathbb{C}/\mathbb{Z}\right) $, by the above lemma, we get that $%
Z\in W(Z)$ is sampling as well.

\begin{lemma}
\label{samp-weak-uniqueness}A separated sequence $Z$ is sampling for $%
\mathcal{F}_{\infty ,\nu }^{\alpha }\left( \mathbb{C}/\mathbb{Z}\right) $ if
and only if every set in $W(Z)$ is a set of uniqueness for $\mathcal{F}%
_{\infty ,\nu }^{\alpha }\left( \mathbb{C}/\mathbb{Z}\right) $.
\end{lemma}

\begin{proof}
Assume there exists a sequence $(F_{n})$ of unit vectors in $\mathcal{F}%
_{\infty ,\nu }^{\alpha }\left( \mathbb{C}/\mathbb{Z}\right) $ such that $%
\Vert F_{n}|Z\Vert \rightarrow 0$ as $n\rightarrow +\infty $. Then by
continuity we can find some sequence in $(\xi_{n})\in \Lambda(\mathbb{Z}) $
such that 
\begin{equation*}
F_{n}(\xi_{n})e^{-\frac{\alpha }{2}|\xi_{n}|^{2}}=\frac{1}{2}\text{.}
\end{equation*}%
Next, for $\xi_{n}=x_{n}+iy_{n}$ let $w_{n}=x_{n}+i\frac{\alpha }{\pi }%
\lfloor \frac{\alpha }{\pi }y_{n}\rfloor $, we consider the function 
\begin{equation*}
G_{n}(z)=e^{-\alpha z\overline{w_{n}}-\frac{\alpha }{2}%
|w_{n}|^{2}}F_{n}(z+w_{n})\text{.}
\end{equation*}%
Note that $G_{n}$ satisfies the functional equation \eqref{funct-equa}.
Furthermore, 
\begin{equation*}
\Vert G_{n}\Vert _{\mathcal{F}_{\infty ,\nu }^{\alpha }\left( \mathbb{C}/%
\mathbb{Z}\right) }=\Vert F_{n}\Vert _{\mathcal{F}_{\infty ,\nu }^{\alpha
}\left( \mathbb{C}/\mathbb{Z}\right) }=1\text{.}
\end{equation*}%
Let $\kappa _{n}=i\frac{\pi }{\alpha }(\frac{\alpha }{\pi }y_{n}-\lfloor 
\frac{\alpha }{\pi }y_{n}\rfloor )$. Then, 
\begin{equation*}
|G_{n}(\kappa _{n})|= |F(\xi_{n})|e^{-\frac{\alpha }{2}|\xi_{n}|^{2}+\frac{%
\alpha }{2}|\kappa_{n}|^{2}}\in \lbrack \frac{1}{2},\frac{e^{\frac{\pi ^{2}}{%
2\alpha }}}{2}]\text{.}
\end{equation*}%
Then by considering a subsequence (if necessarily), we can assume that $%
|G_{n}(\kappa _{n})|$ converges to some $b\in \lbrack \frac{1}{2},\frac{e^{%
\frac{\pi ^{2}}{2\alpha }}}{2}]$. Since $(\kappa _{n})$ is bounded, then we
may also assume that it converge to some $\kappa $ in the segment $[0,i]$.
Finally, by Montel's theorem we can find a subsequence of $(G_{n})$
converging uniformly to some holomorphic function $G $ on every compact
subset of $\mathbb{C}$ satisfying $G\in \mathcal{F}_{\infty ,\nu }^{\alpha
}\left( \mathbb{C}/\mathbb{Z}\right) $ and $G(\kappa )\neq 0$.

Now, let $Z_{n}=Z+w_{n}$, take $Z_{0}$ as the weak limit of $Z_{n}$. If $%
Z_{0}$ is empty then it is not a set of uniqueness for $\mathcal{F}_{\infty
,\nu }^{\alpha }\left( \mathbb{C}/\mathbb{Z}\right) $. If $Z_{0}$ is not
empty, we take $a\in Z_{0}$. Then, for every positive integer $k\geq 1$,
there exists $\zeta _{k}\in Z_{n_k}$ such that $\left\vert a-\zeta
_{k}\right\vert \leq \frac{1}{k}$. By \cite[Lemma 3.1]{seip1}, 
\begin{equation*}
\left\vert e^{-\frac{\alpha }{2}|a|^{2}}\left\vert G_{n_{k}}(a)\right\vert
-e^{-\frac{\alpha }{2}|\zeta _{k}|^{2}}|G_{n_{k}}(\zeta_{k})|\right\vert
\leq c\left\vert a-\xi _{k}\right\vert \leq \frac{1}{k}\rightarrow 0\text{.}
\end{equation*}%
But $\Vert G_{n}|Z_{n}\Vert _{\infty ,\alpha }=\Vert F_{n}|Z\Vert _{\infty
,\alpha }\rightarrow 0$, thus $e^{-\frac{\alpha }{2}|\zeta
_{k}|^{2}}|G_{n_{k}}(\zeta _{k})|\rightarrow 0$, which implies that $G(a)=0$%
. Thus, $G$ vanishes on $Z_{0}$. Since $G\neq 0$ ($G(\kappa)\neq0$) we
obtain that $Z_{0}$ is not a set of uniqueness. This is a contradiction,
since $Z_{n}=Z+w_{n}$ is a set of sampling.
\end{proof}

\begin{lemma}
\label{sub-sam} If $Z$ is sampling for $\mathcal{F}_{\infty ,\nu }^{\alpha
}\left( \mathbb{C}/\mathbb{Z}\right) $, then $Z$ contains a separated
subsequence which is also sampling for $\mathcal{F}_{\infty ,\nu }^{\alpha
}\left( \mathbb{C}/\mathbb{Z}\right) $.
\end{lemma}

\begin{proof}
The proof follows the same approach as in Lemma 4.1 of \cite{seip1}. The
idea is to choose an enumeration of $Z = \{z_1, z_2, \dots\}$ and fix $%
\epsilon > 0$ small (depending on the sampling constant). Set $z_1^{\prime
}= z_1$ and remove all terms from $Z$ within a distance of less than $%
\epsilon$ from $z_1^{\prime }$. Denote the remaining terms by $Z_1 =
\{z_{11}, z_{12}, \dots\}$, preserving the same order. Next, set $%
z_2^{\prime }= z_{11}$ and remove all terms from $Z_1$ within a distance of
less than $\epsilon$ from $z_2^{\prime }$. Denote the remaining terms by $%
Z_2 = \{z_{21}, z_{22}, \dots\}$, preserving the same order. Continue this
process indefinitely if necessary, obtaining a subsequence $Z^{\prime }=
\{z_n^{\prime }\}$ that is $\epsilon$-separated and sampling for a
well-chosen $\epsilon$.
\end{proof}

\begin{lemma}
\label{lemma2} \label{lem2} If $M_{\infty }(Z,\alpha )<\infty $, then $%
M_{\infty }(Z,\alpha +\epsilon )<\infty$ for $\epsilon >0$ sufficiently
small.
\end{lemma}

\begin{proof}
Comparing with the planar setting in Lemma 4.4 of \cite{seip1}, the key
difference here is the loss of translation invariance, since the Weyl
operators $T_a$ do not preserve the functional equation for arbitrary $a \in 
\mathbb{C}$. However, by adapting the proof strategy of Lemma \ref%
{samp-weak-uniqueness}, we obtain the result.
\end{proof}

\begin{proof}[Proof of Theorem \protect\ref{SampInf}]
Without lost of generality, we may assume $\nu =0$. To simplify notations we
denote $\mathcal{F}_{\infty ,\kappa }^{\alpha }\left( \mathbb{C}/\mathbb{Z}%
\right) $ by $\mathcal{F}_{\infty,\kappa }^{\alpha } $ and $\mathcal{F}%
_{2,\kappa }^{\alpha }\left( \mathbb{C}/\mathbb{Z}\right) $ by $\mathcal{F}%
_{2,\kappa}^{\alpha } $. By Lemma \ref{sub-sam}, we may assume that $Z$ is
already separated. By Lemma \ref{lem2}, we have only to show that $%
D^{-}(Z)\geq \frac{\alpha }{\pi }$. Let us assume the contrary, that is 
\begin{equation*}
D^{-}(Z)=\frac{\alpha }{\pi (1+2\epsilon )},\quad \epsilon >0\text{.}
\end{equation*}%
Then, we can find a positive sequence $(r_{n})_{n}$, $r_{n}\rightarrow
+\infty $, and a sequence $(w_{n})_{n}$ in the fundamental domain $\Lambda
(\Gamma )=[0,1)\times\mathbb{R}$ such that 
\begin{equation*}
\frac{n(Z,I_{w_{n},r_{n}})}{r_{n}}<\frac{\alpha }{\pi (1+\epsilon )}\text{.}
\end{equation*}%
Using the notation of definition (\ref{Idisc}), consider $I_{w_{n},r_{n}}$
and $I_{n}=I_{0,r_{n}}$ and set $N:=N(Z,I_{w_{n},r_{n}})=N(Z,I_{n}+\mathrm{Im%
}(w_{n}))$. Then 
\begin{equation*}
\frac{\alpha r_{n}}{\pi (1+2\epsilon )}\leq N<\frac{\alpha r_{n}}{\pi
(1+\epsilon )}\text{.}
\end{equation*}%
Let $\xi _{1},\cdots ,\xi _{N}$ be all the elements of $Z-\mathrm{Im}(w_{n})$
in $I_{n}$. $Z_{1}$ is defined as the set of the $\lfloor \frac{N}{2}\rfloor 
$ elements $\xi _{i}$ with the largest imaginary parts, while $Z_{2}$ is
defined as $Z_{2}=(Z-\mathrm{Im}(w_{n}))\setminus Z_{1}$. Set 
\begin{equation*}
\nu =\kappa _{n}=\frac{\alpha }{\pi }\mathrm{Im}(w_{n})-\lfloor \frac{\alpha 
}{\pi }\mathrm{Im}(w_{n})\rfloor \in \lbrack 0,1]\text{.}
\end{equation*}%
We are now in a position of defining, in explicit form, the following
function $g\in \mathcal{F}_{\infty,\kappa_n }^{\alpha }$ whose zeros are
exactly $\{\xi _{k}\}_{k=1,\cdots ,N}$: 
\begin{equation*}
g(z):=Ce^{\frac{\alpha }{2}z^{2}+2i\pi \kappa _{n}z}\prod_{\xi \in
Z_{1}}(e^{2i\pi z}-e^{2i\pi \xi })\prod_{\xi \in Z_{1}}(e^{-2i\pi
z}-e^{-2i\pi \xi })\text{,}
\end{equation*}%
where $C$ is a normalization constant. Since $Z$ is a finite set, then
trivially $G\in \mathcal{F}_{\infty ,\kappa _{n}}^{\alpha }$. Next important
observation is the possibility of writing $G$ as a linear combination of $%
F_{k}(z)=e^{-\frac{\pi ^{2}}{\alpha }(k+\kappa _{n})^{2}}e^{\frac{\alpha }{2}%
z^{2}+2i\pi (\kappa _{n}+k)z}$ (since these functions are multiples of the
basis functions (\ref{basisF})) 
\begin{equation}
G(z)=\sum_{k=-\lfloor N/2\rfloor +1}^{\lfloor N/2\rfloor
}a_{k}F_{k}(z),\qquad \sum |a_{k}|^{2}=1\text{.}  \label{expansion}
\end{equation}%
The representation (\ref{expansion}) will now be used to obtain a lower
bound for $\Vert G\Vert _{\mathcal{F}_{\infty ,\kappa _{n}}^{\alpha }}$.\
First note that 
\begin{align*}
\int_{I_{n}}|F_{k}(z)|^{2}e^{-\alpha |z|^{2}}dA(z)& =e^{-2\frac{\pi ^{2}}{%
\alpha }(k+\kappa _{n})^{2}}\int_{-\frac{r_{n}}{2}}^{\frac{r_{n}}{2}%
}e^{-4\pi (\kappa _{n}+k)y-2\alpha y^{2}}dy \\
& \geq \int_{-\frac{\pi }{2\alpha }N}^{\frac{\pi }{2\alpha }N}e^{-2\left( 
\frac{\pi }{\sqrt{\alpha }}(\kappa _{n}+k)+\sqrt{\alpha }y\right) ^{2}}dy \\
& \geq \frac{1}{\sqrt{\alpha }}\int_{\frac{\pi }{\alpha }(-\lfloor \frac{N}{2%
}\rfloor +k+\kappa _{n})}^{\frac{\pi }{\alpha }(\lfloor \frac{N}{2}\rfloor
+k+\kappa _{n})}e^{-2y^{2}}dy\geq C\text{.}
\end{align*}%
Inserting (\ref{expansion}) and using the orthogonality of the $F_{k}$'s, 
\begin{equation*}
\int_{I_{n}}|G(z)|^{2}e^{-\alpha |z|^{2}}dA(z)=\sum
|a_{k}|^{2}\int_{I_{n}}|F(z)|^{2}e^{-\alpha |z|^{2}}dA(z)\geq C\text{.}
\end{equation*}%
And we obtain 
\begin{equation*}
\Vert G\Vert _{\mathcal{F}_{\infty ,\kappa _{n}}^{\alpha }}\geq \frac{1}{%
\sqrt{r_{n}}}\left(\int_{I_{n}}|G(z)|^{2}e^{-\alpha
|z|^{2}}dA(z)\right)^{1/2}\geq \sqrt{C}\frac{1}{\sqrt{r_{n}}}\text{.}
\end{equation*}%
Next, for $z=x+iy\in \mathbb{C}/\Gamma $ such that $\left\vert y\right\vert
\geq (1+t)r_{n}/2$, $t>0$, 
\begin{align*}
|F_{k}(z)|^{2}e^{-\alpha |z|^{2}}& =e^{-2\left( \frac{\pi }{\sqrt{\alpha }}%
(\kappa _{n}+k)+\sqrt{\alpha }y\right) ^{2}} \\
& \leq e^{-2\left( \sqrt{\alpha }|y|-\frac{\pi }{\sqrt{\alpha }}(\kappa
_{n}+k)\right) ^{2}} \\
& \leq e^{-2(\frac{\sqrt{\alpha }}{2}tr_{n}-\frac{\pi }{\sqrt{\alpha }}%
\kappa _{n})^{2}}\text{,}
\end{align*}%
for $r_{n}$ sufficiently large. Thus, by (\ref{expansion}) and
Cauchy-Schwarz, 
\begin{equation*}
|G(z)|^{2}e^{-\alpha |z|^{2}}\leq r_{n}e^{-2(\frac{\sqrt{\alpha }}{2}tr_{n}-%
\frac{\pi }{\sqrt{\alpha }}\kappa _{n})^{2}}\text{.}
\end{equation*}%
Now, let $Z_{n}=Z-\mathrm{Im}(w_{n})$ such that 
\begin{equation*}
\Vert G_{n}|Z_{n}\Vert _{\infty ,\alpha }=\sup_{\underset{|\mathrm{Im}%
(z)|\geq r_{n}(1+t)/2}{z\in Z_{n}}}|G(z)|e^{-\frac{\alpha }{2}|z|^{2}}\leq 
\sqrt{r_{n}}e^{-(\frac{\sqrt{\alpha }}{2}tr_{n}-\frac{\pi }{\sqrt{\alpha }}%
\kappa _{n})^{2}}
\end{equation*}%
and consider the following function 
\begin{equation*}
H_{n}(z)=T_{i\mathrm{Im}(w_{n})}G_{n}(z)=e^{-i\alpha z\mathrm{Im}(w_{n})-%
\frac{\alpha }{2}\mathrm{Im}(w_{n})^{2}}G_{n}(z-i\mathrm{Im}(w_{n}))\text{.}
\end{equation*}%
We can easily verify that $H_{n}$ satisfies the functional equation (\ref%
{funct-equa}) with $\nu =0$, and that 
\begin{equation*}
\Vert H_{n}\Vert _{\mathcal{F}_{\infty ,0}^{\alpha }}=\sup_{z\in \Lambda
(\Gamma )}|G_{n}(z-i\mathrm{Im}(w_{n}))|e^{-\frac{\alpha }{2}|z-i\mathrm{Im}%
(w_{n})|^{2}}=\Vert G_{n}\Vert _{\mathcal{F}_{\infty ,\kappa _{n}}^{\alpha
}}\geq C\frac{1}{\sqrt{r_{n}}}\text{.}
\end{equation*}%
We have also 
\begin{equation*}
\Vert H_{n}|Z\Vert _{\infty ,\alpha }=\Vert G_{n}|Z_{n}\Vert _{\infty
,\alpha }\leq \sqrt{r_{n}}e^{-(\frac{\sqrt{\alpha }}{2}tr_{n}-\frac{\pi }{%
\sqrt{\alpha }}\kappa _{n})^{2}}\text{.}
\end{equation*}%
Since $r_{n}\rightarrow +\infty $, we conclude that 
\begin{equation*}
\lim_{n\rightarrow +\infty }\frac{\Vert H_{n}|Z\Vert _{\infty ,\alpha }}{%
\Vert H_{n}\Vert _{\mathcal{F}_{\infty ,0}^{\alpha }}}\leq
C\lim_{n\rightarrow +\infty }r_{n}e^{-(\frac{\sqrt{\alpha }}{2}tr_{n}-\frac{%
\pi }{\sqrt{\alpha }}\kappa _{n})^{2}}=0\text{.}
\end{equation*}%
This contradicts the assumption that $Z$ is sampling for $\mathcal{F}%
_{\Gamma }^{\infty }$.
\end{proof}

Now we will prepare for the proof of Theorem \ref{SampInf}.\ The first step
is the following proposition.

\begin{proposition}
\label{sam-to-uni}If $Z$ is a separated sampling sequence for $\mathcal{F}%
_{2,\nu }^{\alpha }\left( \mathbb{C}/\mathbb{Z}\right)$, then $Z$ is a set
of uniqueness for $\mathcal{F}_{\infty ,\nu }^{\alpha }\left( \mathbb{C}/%
\mathbb{Z}\right) $.
\end{proposition}

\begin{proof}
We may assume that $\nu =0$. Suppose that $Z$ is sampling for $\mathcal{F}%
_{2,\nu }^{\alpha }\left( \mathbb{C}/\mathbb{Z}\right) $, but not a set of
uniqueness for $\mathcal{F}_{\infty ,\nu }^{\alpha }\left( \mathbb{C}/%
\mathbb{Z}\right) $. Then there exists a function in $\mathcal{F}_{\infty
,\nu }^{\alpha }\left( \mathbb{C}/\mathbb{Z}\right) $ not identically zero
and vanishing on $Z$. Then $Z$ is not sampling for $\mathcal{F}_{\infty ,\nu
}^{\alpha }\left( \mathbb{C}/\mathbb{Z}\right) $. According to Lemma \ref%
{lemma2}, $Z$ is not sampling for $\mathcal{F}_{\infty ,\nu }^{\alpha
-\epsilon }\left( \mathbb{C}/\mathbb{Z}\right) $ for $\epsilon $
sufficiently small. Then, there exists a sequence $(F_{n})_{n}$ of unit
vectors in $\mathcal{F}_{\infty ,\nu }^{\alpha -\epsilon }\left( \mathbb{C}/%
\mathbb{Z}\right)$ such that $\Vert F_{n}|Z\Vert _{\infty ,\alpha -\epsilon
}\rightarrow 0$. By Montel's theorem, we can assume that $%
F_{n}(z)\rightarrow F$ on every compact set. Thus we have $F\in \mathcal{F}%
_{\infty ,\nu }^{\alpha -\epsilon }\left( \mathbb{C}/\mathbb{Z}\right)$ and $%
F$ vanish on $Z$. Consider the function 
\begin{equation}
G(z)=e^{\frac{\alpha }{2}z^{2}}e^{-\frac{\alpha -\epsilon }{2}z^{2}}F(z)%
\text{.}
\end{equation}%
Note that $e^{-\frac{\alpha -\epsilon }{2}z^{2}}F(z)$ is $\mathbb{Z}$%
-periodic and hence $G$ verifies the functional equation \eqref{funct-equa}.
Furthermore, since $|F(z)|\leq e^{\frac{\alpha -\epsilon }{2}|z|^{2}}$ we
get 
\begin{equation}
|G(z)|e^{-\frac{\alpha }{2}|z|^{2}}\leq e^{-\epsilon y^{2}},\quad y=\mathrm{%
Im} (z)\text{.}
\end{equation}%
Thus $G\in \mathcal{F}_{2,\nu }^{\alpha }\left( \mathbb{C}/\mathbb{Z}\right) 
$ and vanishes on $Z$, which is a contradiction.
\end{proof}

\begin{proof}[Proof of Theorem \protect\ref{cct}]
We consider a set of sampling $Z$ for $\mathcal{F}_{2,\nu }^{\alpha }\left( 
\mathbb{C}/\mathbb{Z}\right) $. By classical arguments, $Z$ can be written
as a finite union of separated sequences (as stated in Lemma 7.1 \cite{seip1}%
). By Lemma \ref{limits}, $W(Z)$ consists only of sets of sampling. By
Proposition \ref{sam-to-uni}, every set of sampling for $\mathcal{F}_{2,\nu
}^{\alpha }\left( \mathbb{C}/\mathbb{Z}\right) $ is a set of uniqueness for $%
\mathcal{F}_{\infty ,\nu }^{\alpha }\left( \mathbb{C}/\mathbb{Z}\right) $.
Thus, every set in $W(Z)$ is a set of uniqueness for $\mathcal{F}_{\infty
,\nu }^{\alpha }\left( \mathbb{C}/\mathbb{Z}\right) $. By Lemma \ref%
{samp-weak-uniqueness}, $Z$ is a set of sampling for $\mathcal{F}_{\infty
,\nu }^{\alpha }\left( \mathbb{C}/\mathbb{Z}\right) $. By Theorem \ref%
{SampInf}, $Z$ contains a separated sequence $Z^{\prime }$ which is also a
set of sampling and such that $D^{-}(Z^{\prime })>\frac{\alpha }{\pi }$.
\end{proof}

\subsection{Sufficient conditions for sampling}

In this section we will prove the following sufficient condition of the
density theorem for sampling in $\mathcal{F}_{2,\nu }^{\alpha }\left( 
\mathbb{C}/\mathbb{Z}\right) $.

\begin{theorem}
\label{suff} Let $Z$ be a separated sequence of $\mathbb{C}/\mathbb{Z}$. If $%
D^{-}(Z)>\frac{\alpha }{\pi }$, then $Z$ is a set of sampling for $\mathcal{F%
}_{2,\nu }^{\alpha }\left( \mathbb{C}/\mathbb{Z}\right) $.
\end{theorem}

We start by considering a sequence $Z=(z_{n})_{n}$ in $\Lambda(\mathbb{Z})$
which is uniformly close to $\Lambda _{\alpha }:=i\frac{\pi }{\alpha }%
\mathbb{Z}$, in the sense that there exists $Q>0$ such that 
\begin{equation}
\left\vert \mathrm{Im}(z_{k})-\frac{\pi }{\alpha }k\right\vert \leq Q\text{,}
\label{UniClose}
\end{equation}%
for all $k\in \mathbb{Z}$. Define the function 
\begin{equation*}
G(z)=e^{\frac{\alpha }{2}z^{2}}\prod_{k\geq 0}(1-e^{2i\pi (z_{k}-z)})\times
\prod_{k<0}(1-e^{2i\pi (z-z_{k})})\text{.}
\end{equation*}%
The product converges uniformly on every compact subset of $\mathbb{C}$,
thus $G$ is a holomorphic function on $\mathbb{C}$ which satisfies the
functional equation (\ref{funct-equa}). It will play a role in $\mathbb{C}/%
\mathbb{Z}$ similar to the Weierstrass function in $\mathbb{C}$, which has
been used in \cite{seip2}.

\begin{lemma}
\label{estimatesg}Let $Z$ be uniformly close to $\Lambda _{\alpha }$, and $G$
be its associated function defined above. Then 
\begin{equation*}
\left\vert G(z)\right\vert e^{-\frac{\alpha }{2}|z|^{2}}\leq C_{1}e^{\gamma
^{+}|y|}\text{,}
\end{equation*}%
and 
\begin{equation}
\left\vert G(z)\right\vert e^{-\frac{\alpha }{2}|z|^{2}}\geq C_{1}e^{-\gamma
^{-}|y|}d(z,Z)\text{,}  \label{lower-bd}
\end{equation}%
where $\gamma ^{+}$ and $\gamma ^{-}$are two positive constants. Moreover,
for $z_{n}\in Z$,%
\begin{equation*}
\left\vert G^{\prime }(z_{n})\right\vert e^{-\frac{\alpha }{2}%
|z_{n}|^{2}}\geq C_{1}e^{-\gamma ^{-}|\mathrm{Im}(z_{n})|}\text{.}
\end{equation*}
\end{lemma}

\begin{proof}
For $z$ in the compact set%
\begin{equation}
\{z\in \mathbb{C}/\Gamma :\,\left\vert \mathrm{Im}(z)\right\vert \leq 2Q\}%
\text{,}  \label{compact}
\end{equation}%
we have $\left\vert G(z)\right\vert \asymp dist(z,Z)$. Therefore we just
need to consider $\left\vert \mathrm{Im}(z)\right\vert \geq 2Q$. This region
can be naturally divided in two, according to the sign of $\left\vert 
\mathrm{Im}(z)\right\vert $. For $\mathrm{Im}(z)\ge 2Q$, we write%
\begin{equation}
\prod_{k\geq 0}\left( 1-e^{2i\pi (z_{k}-z)}\right) =H_{1}(z)H_{2}(z)H_{3}(z)%
\text{,}  \label{decomposition}
\end{equation}%
where 
\begin{equation*}
H_{1}(z)=\prod_{\mathrm{Im}(z_{k})>\mathrm{Im}(z)+Q}\left( 1-e^{2i\pi
(z_{k}-z)}\right) ,\quad H_{2}(z)=\prod_{|\mathrm{Im}(z_{k}-z)|\leq Q}\left(
1-e^{2i\pi (z_{k}-z)}\right)
\end{equation*}%
and 
\begin{equation*}
H_{3}(z)=\prod_{\underset{k\ge0}{\mathrm{Im}(z_{k})< \mathrm{Im}(z)-Q}%
}\left( 1-e^{2i\pi (z_{k}-z)}\right) \text{.}
\end{equation*}%
We have 
\begin{equation*}
\left\vert H_{1}(z)\right\vert =\exp \left( \sum_{\mathrm{Im}(z_{k})>\mathrm{%
Im}(z)+Q}\log |1-e^{2i\pi (z_{k}-z)}|\right) \asymp \exp \left(\pm \sum_{%
\mathrm{Im}(z_{k})>\mathrm{Im}(z)+Q}e^{-2\pi \mathrm{Im}(z_{k}-z)}\right)
\asymp 1\text{.}
\end{equation*}%
For $H_{2}$, we note that there exists a fixed number $N$ (independent of $z$%
) such that the number of points of $Z $ in $\{z \in \mathbb{C}/\Gamma : |%
\mathrm{Im}(z_k) - \mathrm{Im}(z)| \leq Q\} $ is always bounded above by $N$%
. Let $z_{k_{0}}$ be the point of $Z$ closest to $z$. Then 
\begin{equation*}
\left\vert H_{2}(z)\right\vert =\frac{\left( 1-e^{2i\pi
(z_{k_{0}}-z)}\right) }{dist(z,Z)}\left(\prod_{\underset{k\ne k_0}{|\mathrm{%
Im}(z_{k}-z)|\leq Q}}(1-e^{2i\pi (z_{k}-z)})\right)\times dist(z,Z)\asymp
dist(z,Z)\text{.}
\end{equation*}%
In fact, for the upper bound, 
\begin{equation*}
\left\vert \prod_{\underset{k\neq k_0}{\left\vert \mathrm{Im}%
(z-z_{k})\right\vert \leq Q}}(1-e^{2i\pi (z_{k}-z)})\right\vert \leq \prod_{%
\underset{k\neq k_0}{\left\vert \mathrm{Im}(z-z_{k})\right\vert \leq Q}%
}(1+e^{-2\pi \mathrm{Im}(z_{k}-z)})\leq (1+e^{2\pi Q})^{N}\text{.}
\end{equation*}%
For the lower bound, we add a few additional factors into $(1-e^{2i\pi
(z_{k_{0}}-z})$, if necessary, we may assume $\left\vert \mathrm{Im}%
(z-z_{k})\right\vert >\delta _{0}>0$, $k\geq 0$ and $k\neq k_0$. Then, we
get 
\begin{equation*}
\left\vert \prod_{\underset{k\neq k_0}{\left\vert \mathrm{Im}%
(z-z_{k})\right\vert \leq Q}}(1-e^{2i\pi (z_{k}-z)})\right\vert \geq \prod_{%
\underset{k\neq k_0}{\left\vert \mathrm{Im}(z-z_{k})\right\vert \leq Q}%
}\left\vert 1-e^{-2\pi \mathrm{Im}(z_{k}-z)}\right\vert \geq \min
\{(1-e^{-2\pi \delta _{0}})^{N},(e^{2\pi \delta _{0}}-1)^{N}\}.
\end{equation*}

For $H_{3}$, 
\begin{equation*}
\left\vert H_{3}(z)\right\vert =\prod\limits_{\substack{ \mathrm{Im}{\small %
(z}_{k}{\small -z)<-Q}  \\ {\small k\geq 0}}}\left\vert 1-e^{2i\pi
(z_{k}-z)}\right\vert \text{.}
\end{equation*}%
Now, note that, for $\mathrm{Im} (z_{k} -z)<-Q $ and $k\geq 0$, we can find
two absolute constants $c$ and $C$ such that 
\begin{equation*}
ce^{2\pi \mathrm{Im}(z-z_{k})}\leq \left\vert 1-e^{2i\pi
(z_{k}-z)}\right\vert \leq Ce^{2\pi \mathrm{Im}(z-z_{k})}\text{.}
\end{equation*}%
Thus, 
\begin{equation*}
\left\vert H_{3}(z)\right\vert \leq \prod_{_{\substack{ \mathrm{Im}{\small (z%
}_{k}{\small -z)<-Q}  \\ {\small k\geq 0}}}}Ce^{2\pi \mathrm{Im}%
(z-z_{k})}=C^{\#\{k\geq 0\,\text{\ }\mathrm{Im}\left( z_{k}-z\right)
<-Q,\,\}}\exp \left[ \sum_{_{_{\substack{ \mathrm{Im}{\small (z}_{k}{\small %
-z)<-Q;}  \\ {\small k\geq 0}}}}}2\pi \mathrm{Im}(z-z_{k})\right]
\end{equation*}%
Since $Z$ is uniformly close to $\Lambda _{\alpha }:=i\frac{\pi }{\alpha }%
\mathbb{Z}$, and $k\geq 0$, then $\#\{k\geq 0\,\text{\ }\mathrm{Im}\left(
z_{k}-z\right) <-Q,\,\}\le\frac{\alpha}{\pi} \mathrm{Im}(z)$. And we have
also result, 
\begin{equation*}
\sum_{_{\substack{ \mathrm{Im}{\small (z}_{k}{\small -z)<-Q}  \\ {\small %
k\geq 0}}}}2\pi \mathrm{Im}(z-z_{k})\leq \sum_{0\leq k\leq \frac{\alpha }{%
\pi }\mathrm{Im}(z)}2\pi (\mathrm{Im}(z)-\frac{\pi }{\alpha }k+Q)\leq \alpha 
\mathrm{Im}(z)^{2}+C\mathrm{Im}(z).
\end{equation*}%
for some positive constant $C$ as $\mathrm{Im}(z)\rightarrow +\infty $. The
lower bound is obtained by the same way. Combining these estimates we find
two positives constants $\gamma ^{+}$ and $\gamma ^{-}$ such that 
\begin{equation*}
e^{\alpha \mathrm{Im}(z)^{2}-\gamma ^{-}\mathrm{Im}(z)}\leq \left\vert
H_{3}(z)\right\vert \leq e^{\alpha \mathrm{Im}(z)^{2}+\gamma ^{+}\mathrm{Im}%
(z)}\text{.}
\end{equation*}%
Finally, since $\mathrm{Im}(z)\geq 2Q$, we get 
\begin{equation*}
1-e^{-4\pi Q}e^{2\pi \mathrm{Im}(z_{k})}\leq \left\vert 1-e^{2i\pi
(z-z_{k})}\right\vert \leq 1+e^{-4\pi Q}e^{2\pi \mathrm{Im}(z_{k})}
\end{equation*}%
and $Z$ is uniformly close to $\Lambda _{\alpha }$, it holds 
\begin{equation*}
\prod_{k\leq 0}\left\vert 1-e^{2i\pi (z-z_{k})}\right\vert \asymp 1\text{.}
\end{equation*}%
If\textbf{\ }$\mathrm{Im}$\textbf{$($}$z$\textbf{$)\leq -$}$2Q$ write 
\begin{equation*}
\prod_{k<0}\left( 1-e^{2i\pi (z-z_{k})}\right) =\widetilde{H}_{1}(z)%
\widetilde{H}_{2}(z)\widetilde{H}_{3}(z)\text{,}
\end{equation*}%
where 
\begin{equation*}
\widetilde{H}_{1}(z)=\prod_{\mathrm{Im}(z_{k})<\mathrm{Im}(z)-Q}\left(
1-e^{2i\pi (z-z_{k})}\right) ,\quad \widetilde{H}_{2}(z)=\prod_{|\mathrm{Im}%
(z-z_{k})|\leq Q}\left( 1-e^{2i\pi (z-z_{k})}\right)
\end{equation*}%
and 
\begin{equation*}
\widetilde{H}_{3}(z)=\prod_{\mathrm{Im}(z_{k})\geq \mathrm{Im}(z)+Q,\
k<0}\left( 1-e^{2i\pi (z-z_{k})}\right) \text{.}
\end{equation*}%
By the same way as above we obtain 
\begin{eqnarray*}
\left\vert \widetilde{H}_{1}(z)\right\vert &\asymp &1\text{,} \\
\left\vert \widetilde{H}_{2}(z)\right\vert &\asymp &dist(z,Z)\text{,} \\
e^{\alpha \mathrm{Im}(z)^{2}-\gamma ^{-}|\mathrm{Im}(z)|} &\leq &\left\vert 
\widetilde{H}_{3}(z)\right\vert \leq e^{\alpha \mathrm{Im}(z)^{2}+\gamma
^{+}|\mathrm{Im}(z)|}
\end{eqnarray*}%
and%
\begin{equation*}
\prod_{k\geq 0}\left( 1-e^{2i\pi (z_{k}-z)}\right) \asymp 1\text{.}
\end{equation*}%
From the above estimates, for $z\in \mathbb{C}$ it holds 
\begin{equation*}
\left\vert G(z)\right\vert e^{-\frac{\alpha }{2}|z|^{2}}=e^{-\alpha \mathrm{%
Im}(z)^{2}}\left\vert \prod_{k\geq 0}(1-e^{2i\pi (z_{k}-z)})\times
\prod_{k\leq 0}(1-e^{2i\pi (z-z_{k})})\right\vert \leq Ce^{\gamma ^{+}|%
\mathrm{Im}(z)|}
\end{equation*}%
and 
\begin{equation*}
\left\vert G(z)\right\vert e^{-\frac{\pi }{2}|z|^{2}}=e^{-\pi \mathrm{Im}%
(z)^{2}}\left\vert \prod_{k\geq 0}(1-e^{2i\pi (z_{k}-z)})\times \prod_{k\leq
0}(1-e^{2i\pi (z-z_{k})})\right\vert \geq Ce^{-\gamma ^{-}|\mathrm{Im}%
(z)|}dist(z,Z)\text{.}
\end{equation*}%
This finishes the proof.
\end{proof}

As consequence of the above result we have the following sampling theorem,
which can be proved by a residues argument, exactly as in Lemma 3.1 of \cite%
{seip2}.

\begin{corollary}
Let $Z=\{z_{n}\}$ be a sequence which is uniformly close to $\Lambda _{\beta
}$, $\alpha <\beta $, and $G$ be its associated function. Then for $F\in 
\mathcal{F}_{\infty ,\nu }^{\alpha }\left( \mathbb{C}/\mathbb{Z}\right) $ we
have 
\begin{equation*}
F(z)=2i\pi \sum_{k\in \mathbb{Z}}F(z_{k})e^{\frac{1}{2}(\beta -\alpha
)z_{k}^{2}}\frac{G(z)e^{-\frac{1}{2}(\beta -\alpha )z^{2}}}{G^{\prime
}(z_{k})(1-e^{2i\pi (z_{k}-z)})}\text{,}
\end{equation*}%
with uniform convergence on compacts of $\mathbb{C}/\Gamma $.
\end{corollary}

\begin{proof}[Proof of Theorem \protect\ref{suff}]
Suppose that $0<\alpha <\beta $ and $Z$ is a sequence with $D^{-}(Z)=\frac{%
\beta }{\pi }$. Then by the argument of Beurling \cite[p. 356]{Beurling},
there exists a subsequence $Z^{\prime }$ of $Z$ such that $Z^{\prime }$ is
uniformly close to $\Lambda _{\gamma }$ for some $\alpha <\gamma <\beta $
(Beurling's argument requires removing points from $Z$; it can be used here,
since the proofs of Lemma 6.2 and Lemma 6.3 in \cite{seip1} trivially adapt
to this setting, assuring that if one removes a point from a sampling
sequence, what is left remains a sampling sequence). Thus, we may assume
that $Z$ is uniformly close to $\Lambda _{\beta }$ and that $\nu =0$. We
consider the square $\Omega =[0,1[\times \lbrack 0,\frac{\pi }{\alpha }]$.
Since $(\Omega +ik\frac{\pi }{\alpha })_{k\in \mathbb{Z}}$ forms a partition
of $\mathbb{C}/\Gamma $, 
\begin{equation}
\int_{\mathbb{C}/\Gamma }|F(z)|^{2}e^{-\alpha |z|^{2}}dA(z)=\sum_{k\in 
\mathbb{Z}}\int_{\Omega }|T_{ik\frac{\pi }{\alpha }}F(z)|^{2}e^{-\alpha
|z|^{2}}dA(z)\text{.}  \label{dec}
\end{equation}%
Note that $T_{ik\frac{\pi }{\alpha }}$ is a unitary isomorphism of $\mathcal{%
F}_{2,\nu }^{\alpha }\left( \mathbb{C}/\mathbb{Z}\right) $. Note that $Z+ik%
\frac{\pi }{\alpha }$ is uniformly close to $\Lambda _{\beta }$, we write $%
T_{ik\frac{\pi }{\alpha }}F(z)$ according to the above corollary, since $%
\mathcal{F}_{2,\nu }^{\alpha }\left( \mathbb{C}/\mathbb{Z}\right) \subset 
\mathcal{F}_{\infty ,\nu }^{\alpha }\left( \mathbb{C}/\mathbb{Z}\right) $ 
\begin{equation*}
T_{ik\frac{\pi }{\alpha }}F(z)=2i\pi \sum_{n\in \mathbb{Z}}T_{ik\frac{\pi }{%
\alpha }}F(z_{n}+ik\frac{\pi }{\alpha })e^{\frac{1}{2}(\beta -\alpha
)(z_{n}+ik\frac{\pi }{\alpha })^{2}}\frac{G_{k}(z)e^{-\frac{1}{2}(\beta
-\alpha )z^{2}}}{G^{\prime }(z_{n}+ik\frac{\pi }{\alpha })(1-e^{2i\pi
(z_{n}+ik\frac{\pi }{\alpha }-z)})}\text{,}
\end{equation*}%
where $G_{k}$ is the function associated to $Z+ik\frac{\pi }{\alpha }$. Set $%
\epsilon =\beta -\alpha $ and apply Cauchy-Schwarz and the bounds from Lemma %
\ref{estimatesg} yield 
\begin{equation*}
\left\vert T_{ik\frac{\pi }{\alpha }}F(z)\right\vert ^{2}\leq
CH(z)\sum_{n\in \mathbb{Z}}e^{-\epsilon (y_{n}+k\frac{\pi }{\alpha }%
)^{2}+\gamma |y_{n}+k\frac{\pi }{\alpha }|}e^{-\alpha
|z_{n}|^{2}}|F(z_{n})|^{2}\text{,}
\end{equation*}%
with 
\begin{equation*}
H(z)=\sum_{n\in \mathbb{Z}}e^{-\epsilon (y_{n}+k\frac{\pi }{\alpha }%
)^{2}}\left\vert \frac{G_{k}(z)e^{-\frac{1}{2}(\beta -\alpha )z^{2}}}{%
(1-e^{2i\pi (z_{n}+ik\frac{\pi }{\alpha }-z)})}\right\vert ^{2}\text{.}
\end{equation*}%
Here $y_{n}=\mathrm{Im} (z_{n})$. Observing that the integral 
\begin{equation*}
\int_{\Omega }\left\vert \frac{G_{k}(z)e^{-\frac{1}{2}(\beta -\alpha )z^{2}}%
}{(1-e^{2i\pi (z_{n}+ik\frac{\pi }{\alpha }-z)})}\right\vert ^{2}e^{-\alpha
|z|^{2}}dA(z)<\infty
\end{equation*}%
and is independent of $k$ and $n$, the desired estimate follows from (\ref%
{dec}): 
\begin{equation*}
\int_{\mathbb{C}/\Gamma }|F(z)|^{2}e^{-\alpha |z|^{2}}dA(z)\leq C\Vert
F|Z\Vert _{\mathcal{F}_{2,\nu }^{\alpha }\left( \mathbb{C}/\mathbb{Z}\right)
}^{2}\text{.}
\end{equation*}
\end{proof}

\section{Interpolation}

We first show the sufficient condition.

\begin{theorem}
Suppose that $Z$ is a separated sequence. If $D^{+}(Z)<\frac{\alpha }{\pi }$
then $Z$ is interpolating for $\mathcal{F}_{2,\nu }^{\alpha }\left( \mathbb{C%
}/\mathbb{Z}\right) $.
\end{theorem}

\begin{proof}
If $Z$ is separated sequence in $\Lambda(\mathbb{Z})$ with $D^{+}(Z)=\frac{%
\gamma }{\pi }$ and $\gamma <\alpha $. Then we can, if necessary, follow
Beurling's argument \cite[Lemma 4, pg. 353]{Beurling}, and add some points
to expand $Z$ to a separated sequence which is uniformly close to $i\frac{%
\pi }{\beta }\mathbb{Z}$, with $\beta \in (\gamma ,\alpha )$ (as in the
proof of sufficiency for sampling in the previous section, Beurling's
argument can be used, since the proofs of Lemma 6.2 and now Lemma 6.4 in 
\cite{seip1} extend to this setting, assuring that if one adds a point to an
interpolating sequence, what is left remains an interpolating sequence). We
may therefore assume that $Z$ is uniformly close to the sequence $i\frac{\pi 
}{\beta }\mathbb{Z}$, that $\beta <\alpha $ and $\nu =0$. Let $(a_{n})_{n\in 
\mathbb{Z}}$ be a given sequence such that 
\begin{equation*}
a_{k}e^{-\frac{\alpha }{2}|z_{k}|^{2}}\in \ell ^{2}(\mathbb{Z}).
\end{equation*}%
Define the function 
\begin{equation*}
G_{n,Z}:=G_{n}(z)=\left\{ 
\begin{matrix}
\prod_{k\neq n}(1-e^{2i\pi (z_{k}-z)})\times \prod_{k<0}(1-e^{2i\pi
(z-z_{k})}), & \text{ if }n\geq 0 \\ 
\prod_{k\geq n}(1-e^{2i\pi (z_{k}-z)})\times \prod_{k\neq n}(1-e^{2i\pi
(z-z_{k})}), & \text{ if }n<0%
\end{matrix}%
\right. \text{.}
\end{equation*}%
Next, for $z_{n}=x_{n}+iy_{n}$, we write $w_{n}=x_{n}+i\frac{\pi }{\alpha }%
\lfloor \frac{\alpha }{\pi }y_{n}\rfloor $. Then, since $G_{n}(z_{m})=\delta
_{n,m}$ we have 
\begin{equation}
F(z)=\sum_{n\in \mathbb{Z}}a_{n}e^{\alpha (z-z_{n})\overline{w_{n}}+\frac{%
\alpha }{2}\left( (z-w_{n})^{2}-(z_{n}-w_{n})^{2}\right) }\frac{G_{n}(z)}{%
G_{n}(z_{n})}  \label{func-interpolation}
\end{equation}%
satisfies $F(z_{m})=a_{m}$\ for all $z_{n}\in Z$. To see that it is an
interpolation function it just remains to show that $F\in \mathcal{F}_{2,\nu
}^{\alpha }\left( \mathbb{C}/\mathbb{Z}\right)$. Using the estimates of
Lemma \ref{estimatesg}, 
\begin{align*}
\left\vert G_{n}(z)\right\vert =& \left\vert
G_{n,Z-z_{n}}(z-z_{n})\right\vert \\
& \leq C\left\vert e^{-\frac{\beta }{2}(z-z_{n})^{2}}\right\vert \times e^{%
\frac{\beta }{2}|z-z_{n}|^{2}+\gamma ^{+}|\mathrm{Im}(z-z_{n})|} \\
& \leq Ce^{\beta (y-y_{n})^{2}+\gamma ^{+}|y-y_{n}|},\quad y=\mathrm{Im}(z)%
\text{ and }y_{n}=\mathrm{Im}(z_{n})
\end{align*}%
and 
\begin{align*}
\left\vert G_{n}(z_{n})\right\vert =& \left\vert e^{-\frac{\beta }{2}%
z_{n}^{2}}\right\vert \times \left\vert e^{\frac{\beta }{2}%
z_{n}^{2}}G_{n}(z_{n})\right\vert \\
& \geq C\left\vert e^{-\frac{\beta }{2}z_{n}^{2}}\right\vert \times e^{\frac{%
\beta }{2}|z_{n}|^{2}-\gamma ^{-}|\mathrm{Im}(z_{n})|}d(z_{n},Z\setminus
\{z_{n}\} \\
& \geq C\delta e^{\beta y_{n}^{2}-\gamma ^{-}|y_{n}|}\text{,}
\end{align*}%
where $\delta $ is the separation constant of $Z$. Using these estimates, 
\begin{align*}
\left\vert F(z)\right\vert e^{-\frac{\alpha }{2}|z|^{2}}& \leq \sum
|a_{n}|e^{-\frac{\alpha }{2}|w_{n}|^{2}}\left\vert e^{-\frac{\alpha }{2}%
|z-w_{n}|^{2}+\frac{\alpha }{2}(z-w_{n})^{2}}\right\vert \times \left\vert
e^{-\alpha z_{n}\overline{w_{n}}+\alpha |w_{n}|^{2}-\frac{\alpha }{2}%
(z_{n}-w_{n})^{2}}\times \frac{G_{n}(z)}{G_{n}(z_{n})}\right\vert \\
& \leq C\sum |a_{n}|e^{-\frac{\alpha }{2}|w_{n}|^{2}}e^{-(\alpha -\beta
)(y-y_{n})^{2}+\sigma _{1}|y|+\sigma _{2}|y_{n}|}\text{,}
\end{align*}%
for some positive constants $\sigma _{1}$ and $\sigma _{2}$. Since $Z$ is
uniformly close to the sequence $i\frac{\beta }{\pi }\mathbb{Z}$, we can
find another positive constant $\lambda $ such that 
\begin{equation*}
\left\vert F(z)\right\vert e^{-\frac{\alpha }{2}|z|^{2}}\leq \sum |a_{n}|e^{-%
\frac{\alpha }{2}|w_{n}|^{2}}e^{-\lambda (y-\frac{\beta }{\pi }n)^{2}}\text{.%
}
\end{equation*}%
This show that the series defined in \eqref{func-interpolation} converges
uniformly on every compact set to an entire function $F$. Furthermore, one
checks that $F$ satisfies the functional equation \eqref{funct-equa}, with $%
\nu =0$, and of finite norm in $\mathcal{F}_{2,\nu }^{\alpha }\left( \mathbb{%
C}/\mathbb{Z}\right)$. Thus, $F\in \mathcal{F}_{2,\nu }^{\alpha }\left( 
\mathbb{C}/\mathbb{Z}\right)$ and, by construction, $F(z_{n})=a_{n}$ for all 
$z_{n}\in Z$.
\end{proof}

We now show the necessary part of the interpolation result. We begin by
recalling some results which are straightforward analogues of the Fock space
case.

\begin{lemma}
If $Z$ is an interpolating sequence for $\mathcal{F}_{2,\nu }^{\alpha
}\left( \mathbb{C}/\mathbb{Z}\right) $, then it is separated.
\end{lemma}

\begin{proof}
As in \cite[Lemma 5.1]{seip1}.
\end{proof}

\begin{lemma}
\label{weak-cv} Suppose that $Z_{n}\rightarrow Z$ weakly. Then $N_{2
}(Z)\leqq \lim \inf N_{2}\left( Z_{n}\right) $.
\end{lemma}

\begin{proof}
As in \cite[Lemma 5.2]{seip1}.
\end{proof}

For fixed $z$, consider%
\begin{equation*}
\varrho _{2}(z,Z)=\sup_{F}e^{-\frac{\alpha }{2}|z|^{2}}|F(z)|
\end{equation*}%
where the supremum is taken over all functions $F$ for which $F(\zeta )=0$, $%
\zeta \in Z$, and $\Vert F\Vert _{\mathcal{F}_{2,\nu }^{\alpha }\left( 
\mathbb{C}/\mathbb{Z}\right) }\leq 1$.

\begin{lemma}
\label{lem-sup} If $Z$ is an interpolating sequence for $\mathcal{F}_{2,\nu
}^{\alpha }\left( \mathbb{C}/\mathbb{Z}\right) $, then $\varrho
_{2}(z_{0},Z)>0$ when $z_{0}\notin Z$.
\end{lemma}

\begin{proof}
Since $Z$ is an interpolating sequence, then it is not a set of uniqueness.
Take $F\in \mathcal{F}_{2,\nu }^{\alpha }\left( \mathbb{C}/\mathbb{Z}\right) 
$ not identically zero and vanishing on $Z$. Then by dividing by a non
negative power of $e^{2i\pi z}-e^{2i\pi z_{0}}$, which preserves membership
in $\mathcal{F}_{2,\nu }^{\alpha }\left( \mathbb{C}/\mathbb{Z}\right) $, we
arrive at the desired result.
\end{proof}

\begin{lemma}
For $z_{0}\notin Z$, we have 
\begin{equation*}
N_{2 }\left( Z\cup \left\{ z_{0}\right\} \right) \leq \frac{1+2N_{2 }(Z)}{%
\varrho _{2 }\left( z_{0},Z\right) }\text{.}
\end{equation*}
\end{lemma}

\begin{proof}
As in the case of the Fock space \cite[Lemma 8.3]{seip1}.
\end{proof}

\begin{lemma}
Given $\delta _{0},\ell _{0}$, and $\alpha $, there exists a positive
constant $C=C\left( \delta _{0},\ell _{0},\alpha \right) $ such that if $%
N_{2 }(Z)\leqq \ell _{0}$ and $d(z,Z)\geqq \delta _{0}$, then 
\begin{equation*}
\varrho _{2 }(z,Z)\geq C\text{.}
\end{equation*}
\end{lemma}

\begin{proof}
Comparing with the planar setting, the main difference here is the loss of
translation invariance. Nevertheless, we can still adapt the same idea to
obtain the result. We argue by contra-positive: assume the existence of a
sequence $Z_{n}$ of interpolating sets for $\mathcal{F}_{2,\nu }^{\alpha
}\left( \mathbb{C}/\mathbb{Z}\right) $ and a sequence $z_{n}$ of points in $%
\mathbb{C}$ such that 
\begin{equation*}
N_{2}(Z_{n})\leq \ell _{0},\text{ and }d(z_{n},Z_{n})\geq \delta _{0}
\end{equation*}%
and $\varrho _{2}(z_{n},Z_{n})\rightarrow 0$. Since the Weyl operator $T_{a}$
with $\mathrm{Im}(a)\in \frac{\pi }{\alpha }\mathbb{Z}$ preserves $\mathcal{F%
}_{2,\nu }^{\alpha }\left( \mathbb{C}/\mathbb{Z}\right) $, we may assume
that, $z_{n}$ belongs to the set $[0,1]\times \lbrack 0,\frac{\pi }{\alpha }%
] $. Thus, by considering a subsequence if necessary, we may assume that $%
z_{n} $ converge to some $z_{0}$. By considering a subsequence, we may also
assume that $Z_{n}$ converges weakly to $Z^{\prime }$. By Lemma \ref{weak-cv}%
, $N_{2}(Z^{\prime })\leq \ell _{0}$ and $d(z_{0},Z^{\prime })\geq \delta
_{0}.$ The rest of the proof follows the proof of \cite[Lemma 5]{Beurling}.
However, we provide the rest of the proof for convenience. By Lemma \ref%
{lem-sup}, there exists a function $f \in \mathcal{F}_{2,\nu }^{\alpha
}\left( \mathbb{C}/\mathbb{Z}\right) $ which vanishes on $Z^{\prime }$, $%
|f(z_0)|e^{-\frac\alpha2|z_0|^2} = A> 0$ and $\Vert F\Vert _{\mathcal{F}%
_{2,\nu }^{\alpha }\left( \mathbb{C}/\mathbb{Z}\right) }\leq 1$.

Set $\varepsilon_n = \|f|Z_n\|_{2,\alpha}$. This sequence converges to $0$.
Since $Z_n$ is interpolating sequence, we can find a functions $g_n \in 
\mathcal{F}_{2,\nu }^{\alpha }\left( \mathbb{C}/\mathbb{Z}\right) $ such
that $g_n = f $ on $Z_n $ and $\|g_n\|_{\mathcal{F}_{2,\nu }^{\alpha }\left( 
\mathbb{C}/\mathbb{Z}\right)} \leq \ell_0 \varepsilon_n$. Consider the
function 
\begin{equation*}
f_n(z) = \frac{f(z) - g_n(z)}{\|f\|_{\mathcal{F}_{2,\nu }^{\alpha }\left( 
\mathbb{C}/\mathbb{Z}\right)} + \ell_0 \varepsilon_n}.
\end{equation*}
We have $\|f_n\|_{2,\alpha} \leq 1$ and $f_n $ vanishes on $Z_n$. Since $%
|g_n(z_n)|e^{-\frac\alpha2|z_n|^2} \leq \|g_n\|_{\mathcal{F}_{2,\nu
}^{\alpha }\left( \mathbb{C}/\mathbb{Z}\right)} \leq \ell_0 \varepsilon_n
\to 0$ we get $\rho_2(z_n, Z_n) \geq |f_n(z_n)|e^{-\frac\alpha2|z_n|^2} \to 
\frac{A}{\|f\|_{p,\alpha}} > 0.$
\end{proof}

Finally we recall the following lemma, proved in \cite[Lemma 8.4]{seip1}.

\begin{lemma}
\label{lowerInt}For given $\ell _{0}$ and $\alpha $, there exists a positive
constant $C=C\left( \ell _{0},\alpha \right) $ so that if $N_{2 }(Z,\alpha
)\leqq \ell _{0}$, then for every square $Q$ with $\left\vert Q\right\vert
\geq 1$, 
\begin{equation*}
\iint_{Q}\log \varrho _{2 }(z,Z)dxdy\geq -C|Q|^{2}\text{.}
\end{equation*}
\end{lemma}

\begin{theorem}
Suppose that $Z$ is an interpolating set for $\mathcal{F}_{2,\nu }^{\alpha
}\left( \mathbb{C}/\mathbb{Z}\right) $. Then $Z$ is separated and $D^{+}(Z)<%
\frac{\alpha }{\pi }.$
\end{theorem}

\begin{proof}
Consider an arbitrary rectangle of the form $\mathcal{R}=[0,1)\times \lbrack 
\mathrm{Im}(w)-\frac{R}{2},\mathrm{Im}(w)+\frac{R}{2}]$ and associate with
the rectangle $\mathcal{R}$ the square $\mathcal{Q}$ of side length $R$,
which contains $\mathcal{R}$ in the middle. To adapt Seip's idea in \cite[p.
100]{seip1}, divide $\mathcal{Q}$ into $\lfloor R\rfloor \times \lfloor
R\rfloor $ squares, denoted by $Q_{j}$. Since $Z$ is separated, we can find
a point $z_{j}$ with $d\left( z_{j},Z\right) \geq \delta _{0}=\delta
_{0}\left( N_{2}(Z),\alpha \right) $. Let $Z_{j}=Z\cup \left\{ z_{j}\right\} 
$. By Lemma \ref{lowerInt}, 
\begin{equation*}
\iint_{Q_{j}}\log \varrho _{2}\left( z,Z_{j}\right) dxdy\geq -C(l,\alpha )%
\text{.}
\end{equation*}%
Thus, for $z\in Q_{j}$ and $z\not\in Z_j$, we can find $F\in \mathcal{F}%
_{2,\nu }^{\alpha }\left( \mathbb{C}/\mathbb{Z}\right) $ vanishing on $Z_{j}$%
, with $\Vert F\Vert _{\mathcal{F}_{2,\nu }^{\alpha }\left( \mathbb{C}/%
\mathbb{Z}\right) }\leqq 1$ and $F(z)=\varrho _{\infty }\left( z,Z\right) $.

Now, noticing that for $F\in \mathcal{F}_{2,\nu }^{\alpha }\left( \mathbb{C}/%
\mathbb{Z}\right) \subset \mathcal{F}_{\infty ,\nu }^{\alpha }\left( \mathbb{%
C}/\mathbb{Z}\right) $, we have from (\ref{normequality}), $\Vert F\Vert _{%
\mathcal{F}_{\infty ,\nu }^{\alpha }\left( \mathbb{C}/\mathbb{Z}\right)
}=\Vert F\Vert _{\mathcal{F}_{\infty }^{\alpha }\left( \mathbb{C}\right) }$
and since $F$ is entire, we conclude that $F\in \mathcal{F}_{\infty
}^{\alpha }\left( \mathbb{C}\right) $, where $\mathcal{F}_{\infty }^{\alpha
}\left( \mathbb{C}\right) $ is the classical Fock space with norm $\Vert
F\Vert _{\infty }=\sup_{z\in \mathbb{C}}|F(z)|e^{-\frac{\alpha }{2}|z|^{2}}$%
. Set $Z^{\prime }=Z+\mathbb{Z}$. Since $\mathcal{F}_{\infty }^{\alpha
}\left( \mathbb{C}\right) $ is invariant under translations, this gives a
function $F\in \mathcal{F}_{\infty }^{\alpha }\left( \mathbb{C}\right) $,
such that $F$ vanishes on $Z^{\prime }-z$, with $\Vert F\Vert _{\mathcal{F}%
_{\infty }^{\alpha }\left( \mathbb{C}\right) }\leq 1$ and $F(0)=\varrho
_{2}\left( z,Z\right) $.

Since $F\in \mathcal{F}_{\infty }^{\alpha }\left( \mathbb{C}\right) $ is an
entire function, we can now apply Jensen's formula on $\left\vert
z\right\vert <r$\ and repeat Seip's argument \cite[p. 100]{seip1}. Denoting
by $n\left( Q^{-}\right) $ the number of points from $Z$\ contained in $%
Q^{-} $, we find 
\begin{equation*}
\frac{n\left( Q^{-}\right) }{(R-2r)^{2}}\leq \left( \frac{\alpha }{\pi }-%
\frac{2}{\pi }\frac{\log r}{r^{2}}+\frac{C(l,\alpha )}{r^{2}}\right) \left(
1-\frac{2r}{R}\right) ^{-2}\text{.}
\end{equation*}%
As in \cite{seip1}, $Q^{-}$ is the square of side length $R-2r$, consisting
of those points whose distance the complement of $\mathcal{Q}$ exceeds $r$.
Denote by $R^{-}$ be the rectangle of length $R-2r$ contained in $\mathcal{R}
$. since every point is computed at last $R-2r$ times. We obtain 
\begin{equation*}
\frac{n\left( R^{-}\right) }{R-2r}\leq \left( \frac{\alpha }{\pi }-\frac{2}{%
\pi }\frac{\log r}{r^{2}}+\frac{C(l,\alpha )}{r^{2}}\right) \left( 1-\frac{2r%
}{R}\right) ^{-2}\text{.}
\end{equation*}%
Observing the possibility of choosing $r$ so large that 
\begin{equation*}
\frac{\alpha }{\pi }-\frac{2}{\pi }\frac{\log r}{r^{2}}+\frac{C(l,\alpha )}{%
r^{2}}<\frac{\alpha }{\pi }\text{,}
\end{equation*}%
we can let $R\rightarrow \infty $, and conclude that $D^{+}(Z)<\alpha /\pi $.
\end{proof}

\end{document}